\documentclass[11pt,a4paper,reqno]{amsart}
\usepackage{amsfonts,caption,amsmath,amssymb,enumerate,latexsym,xcolor,amsthm,multicol,cases,tikz,empheq}
\usepackage{hyperref}
\usepackage[normalem]{ulem}
\usepackage{array} % Include in the preamble
\usetikzlibrary{calc}

\newcommand\Aut{\mathrm{Aut}}
\newcommand\BiCos{\mathrm{BiCos}}\newcommand{\bfc}{c}
\newcommand\calO{{\mathcal O}}\newcommand{\calS}{\mathcal{S}}\newcommand\Center{\mathrm{Z}}\newcommand{\Cen}{\mathrm{C}}
\newcommand\Cay{\mathrm{Cay}}\newcommand{\Core}{\mathrm{Core}}\newcommand{\calI}{\mathcal{I}}

\newcommand\Fix{\mathrm{Fix}}

\newcommand\Hol{\mathrm{Hol}}

\newcommand\la{\langle}

\newcommand{\maxsub}{<_{\mathrm{max}}}
\newcommand\Nor{\mathrm{N}}

\newcommand\PGL{\mathrm{PGL}}
\newcommand\PSL{\mathrm{PSL}}

\newcommand\ra{\rangle}

\newcommand\Sym{\mathrm{Sym}}

\newtheorem{theorem}{Theorem}[section]
\newtheorem{corollary}[theorem]{Corollary}
\newtheorem{lemma}[theorem]{Lemma}
\newtheorem{proposition}[theorem]{Proposition}

\theoremstyle{definition}
\newtheorem{definition}[theorem]{Definition}

\newtheorem{conjecture}[theorem]{Conjecture}

\newtheorem*{remark}{Remark}

\begin{document}
\title[Asymptotic stability of Cayley graphs on abelian groups]{Asymptotic stability of Cayley graphs on abelian groups}

\author[B.~Xia]{Binzhou Xia}
\address[Binzhou Xia]{School of Mathematics and Statistics, The University of Melbourne, Parkville, VIC 3010, Australia}
\email{binzhoux@unimelb.edu.au}

\author[Z.~Zhang]{Zhishuo Zhang}
\address[Zhishuo Zhang]{School of Mathematics and Statistics, Hainan University, Haikou, P. R. China}
\email{zhishuozhang@hainanu.edu.cn}

\author[S.~Zheng]{Shasha Zheng}
\address[Shasha Zheng]{Department of Algebra and Geometry, Comenius University, 842 48, Bratislava, Slovakia}
\email{shasha.zheng@fmph.uniba.sk}

\begin{abstract}
For a finite group $G$, we say that a Cayley graph $\Gamma$ on $G$ is a most rigid representation (MRR) of $G$ if its full automorphism group has the smallest possible order among all Cayley graphs on $G$, and say that $\Gamma$ is stable if every automorphism of $\Gamma\times K_2$ comes from $\mathrm{Aut}(\Gamma)\times\Aut(K_2)$. Although the study of stability has attracted significant attention, particularly regarding Cayley graphs on abelian groups, a complete classification is currently out of reach even for Cayley graphs on cyclic groups. In this paper, we prove that almost all Cayley graphs on finite abelian groups are stable MRRs. This strengthens the main result of Dobson, Spiga and Verret [Combinatorica, 36 (2016), no.~4, 371--393], which states that almost all Cayley graphs on finite abelian groups are MRRs.

\end{abstract}

\maketitle

\section{Introduction}

Let $G$ be a finite group and $S$ an inverse-closed subset of $G$. The \emph{Cayley graph} of $G$ with \emph{connection set} $S$, denoted by $\Cay(G,S)$, is the graph with vertex set $G$ such that $x$ and $y$ are adjacent if and only if $yx^{-1}\in S$. Note that the right regular representation of $G$, denoted by $R(G)$, acting on $G$ by translations, is always a subgroup of $\Aut(\Cay(G,S))$. Hence the full automorphism group of any Cayley graph on $G$ has order at least $|G|$. To provide a formal measure of the lowest possible level of symmetry among Cayley graphs of $G$, Imrich and  Watkins~\cite{Imrich1976} introduced the \emph{Cayley index}, defined as
\[
\kappa(G)=\min_{1\notin S\subseteq G,\,S=S^{-1}}|\Aut(\Cay(G,S))|/|G|.
\]
{Since adding the identity of $G$ to $S$ or removing it from $S$ does not change $\Aut(\Cay(G,S))$, the minimum may equivalently be taken over all inverse-closed subsets $S$ of $G$.}
A Cayley graph of $G$ is called a \emph{most rigid representation} (MRR for short) of the group if it attains this minimum. Groups with $\kappa(G)=1$ are of special interest, and in this case an MRR of $G$ is called a \emph{graphical regular representation} (GRR for short) of the group. The main results of~\cite{DSV2016,MSV2015,XZ2023} together show that almost all finite Cayley graphs are MRRs.

For a graph $\Gamma$, the \emph{standard double cover} (also called the \emph{canonical double cover}, \emph{bipartite double cover} or \emph{Kronecker double cover}) of $\Gamma$, denoted by $D(\Gamma)$, is the direct
product $\Gamma\times K_2$ of $\Gamma$ with $K_2$, the complete graph on two vertices. The vertex set of $D(\Gamma)$ is $V(\Gamma)\times \{0,1\}$, with vertices $(u,x)$ and $(v,y)$ being adjacent in $D(\Gamma)$ if and only if $u$ and $v$ are adjacent in $\Gamma$ and $x\neq y$. It is clear that\[
    \Aut(D(\Gamma))\geq \Aut(\Gamma)\times \Aut(K_2)=\Aut(\Gamma)\times C_2.
\]
An automorphism of $D(\Gamma)$ not contained in the subgroup $\Aut(\Gamma)\times \Aut(K_2)$ is called \emph{unexpected}. A graph $\Gamma$ is said to be \emph{stable} if its standard double cover $D(\Gamma)$ admits no unexpected automorphisms, and \emph{unstable} otherwise.

This concept of stability was first introduced by Maru\v si\v c, Scapellato and Zagaglia Salvi~\cite{MSZ1989} in 1989 and then studied by various researchers. Disconnected graphs, bipartite graphs admitting nontrivial automorphisms, and graphs containing distinct vertices with the same neighbourhoods are easily seen to be unstable; these are usually referred to as \emph{trivially unstable} graphs. Beyond these cases, various criteria for stability and instability have been developed; see, for example,~\cite{HM2024,S2001,SW2008}. The unstable graphs have also been classified in several infinite families, including generalized Petersen graphs~\cite{QXZ2018}, toroidal and triangular grids~\cite{DW2023}, and Rose Window graphs~\cite{AKK2024}.

Beyond these specific families, significant attention has been devoted to Cayley graphs, with a particular emphasis on the abelian case. The stability of these graphs has been investigated in several cases, notably for groups of odd order. Qin, Xia and Zhou~\cite{QXZ2019} proved that, apart from the trivial exceptions, Cayley graphs on cyclic groups of prime order are stable. This was extended to cyclic groups of arbitrary odd order by Hujdurovi\'c and Fernandez~\cite{FH2022}, and subsequently to arbitrary finite abelian groups of odd order by Morris~\cite{Morris2021}. In the even order case, circulants of order twice a prime were classified by Hujdurovi\'c, Mitrovi\'c and Morris~\cite{HMM2021}. However, a complete classification remains out of reach even for cyclic groups~\cite{HMM2023}, and for general abelian groups of even order, further criteria for instability continue to be developed~\cite{Bychawski2024,HK2023,Zhang2025}.

These results indicate that the classification of unstable Cayley graphs is a subtle problem even for relatively restricted classes. Beyond the problem of complete classification, one may consider stability from an asymptotic point of view and roughly expect that the symmetry structure would rarely be enlarged unexpectedly by passing to the canonical double cover. This led the first author to propose the following conjecture during the 2024 Workshop on Group Actions on Discrete Structures in Kranjska Gora.

\begin{conjecture}
For a group $G$ of order $r$, the proportion of inverse-closed subsets $S$ of $G$ such that $\Cay(G,S)$ is stable approaches $1$ as $r\to\infty$.
\end{conjecture}

To approach this conjecture, we focus on MRRs. Since almost all Cayley graphs are MRRs~\cite{DSV2016,MSV2015,XZ2023}, it would suffice to show that almost all MRRs are stable. It is easy to see that if $\Gamma$ is a Cayley graph of a group $G$, then its standard double cover $D(\Gamma)$ is a Cayley graph of $G\times C_2$. Note that whenever $\kappa(G)=\kappa(G\times C_2)$, a Cayley graph $\Gamma$ of $G$ is a stable MRR if and only if its standard double cover $D(\Gamma)$ is an MRR of $G\times C_2$. In particular, this holds for sufficiently large groups $G$. Thus, the asymptotic enumeration of stable MRRs can be viewed equivalently as the enumeration of MRRs among the standard double covers of Cayley graphs.

Although almost all Cayley graphs of $G\times C_2$ are known to be MRRs when $G$ is sufficiently large, the standard double covers form only a structured subfamily of these Cayley graphs. Determining the relevant proportion within this subfamily is therefore a separate problem.

In this paper, we focus on the abelian case. For any sufficiently large abelian group $G$ (of order greater than 27), the Cayley index is at most 2, and any MRR $\Gamma$ of $G$ has full automorphism group $R(G) \rtimes \langle \iota \rangle$, where $\iota$ denotes the inversion on $G$.  In this case, the standard double cover $D(\Gamma)$ is an MRR of $G \times C_2$ if and only if
\[
\Aut(D(\Gamma))=(R(G) \rtimes \langle \iota \rangle) \times C_2.
\]
We prove the conjecture for abelian groups by showing that almost all such standard double covers have the smallest possible automorphism groups; see Theorems~\ref{thmNumber} and~\ref{thmIso} and the corresponding corollaries.

For each $\delta\in(0,1/2)$ and $r>0$, let
\begin{equation}\label{eqHdelta}
    {h_\delta(r)= 2^{-\frac{r}{24}+(r^{2\delta}+r^\delta+38)(\log_2r)^2+9\log_2r+4}+2^{-\frac{r^\delta}{25}+(\log_2r)^2+3\log_2r+1}.}
\end{equation}
Clearly, $h_\delta(r)$ approaches $0$ as $r\to \infty$. The main result of our paper is as follows.

\begin{theorem}\label{thmNumber}
    Let $G$ be an abelian group of order $r$ and of exponent greater than $2$, let $\iota$ be the inversion on $G$, and let $\delta\in(0,1/2)$. Then the proportion of inverse-closed subsets $S\subseteq G$ with
    \[
        \Aut(D(\Cay(G,S)))=(R(G)\rtimes\la\iota\ra)\times C_2
    \]
    is at least $1-h_\delta(r)$, where $h_\delta(r)$ is as in~\eqref{eqHdelta}. In particular, the proportion of inverse-closed subsets $S\subseteq G$ such that $\Cay(G,S)$ is a stable MRR of $G$ is at least $1-h_\delta(r)$.
\end{theorem}

% [2026-09-14] Remark removed as suggested: with the revised h_delta(r) it gives no effective numerical bound.
%\begin{remark}
%    Take $\delta=0.001$. A routine calculus argument shows that whenever $r\geq 50000$,
%    \[
%        2^{-\frac{r}{24}+(r^{2\delta}+r^\delta+6)(\log_2r)^2+5\log_2r+4}<2^{-\frac{2}{25}r^\delta+3\log_2r+1}.
%    \]
%    Thus, for $r\geq 50000$, the proportion of inverse-closed subsets $S\subseteq G$ with
%    \[
%        \Aut(D(\Cay(G,S)))=(R(G)\rtimes\la\iota\ra)\times C_2
%    \]
%    is at least $1-h_\delta(r)>1-2^{-\frac{2}{25}r^{0.001}+3\log_2r+2}$.
%\end{remark}
%\par\begingroup[Pending. With the revised~\eqref{eqHdelta}, the displayed inequality holds for $r\geq 320000$ (the exact threshold is $r\geq 315320$) rather than $r\geq 50000$, and the two summands change accordingly. Note also that, for $\delta=0.001$, the second summand of $h_\delta(r)$ is less than $1$ only for extremely large $r$, so this remark gives no effective numerical bound. Whether to update or remove it is to be decided. {suggest removing}]\par\endgroup

Combining Theorem~\ref{thmNumber} with~\cite[Theorem~1.2(b)]{GSX2025}, we obtain the following corollary.

\begin{corollary}\label{cor}
    Let $G$ be an abelian group of order $r$, and let $\iota$ be the inversion on $G$. Then the proportion of inverse-closed subsets $S\subseteq G$ with
    \[
    \Aut(D(\Cay(G,S)))=(R(G)\rtimes\la\iota\ra)\times C_2
    \]
    approaches $1$ as $r\to \infty$. In particular, the proportion of inverse-closed subsets $S\subseteq G$ such that $\Cay(G,S)$ is a stable MRR of $G$ approaches $1$ as $r\to \infty$.
\end{corollary}

As a direct consequence, Corollary~\ref{cor} recovers the following result of Dobson, Spiga and Verret, which states that almost all Cayley graphs of finite abelian groups are MRRs:

\begin{corollary}[{\cite[Theorem~1.5]{DSV2016}}]\label{corGRR}
    Let $G$ be an abelian group of order $r$, and let $\iota$ be the inversion on $G$. Then the proportion of inverse-closed subsets $S\subseteq G$ with $\Aut(\Cay(G,S))=R(G)\rtimes\la\iota\ra$ approaches $1$ as $r\to \infty$.
\end{corollary}

An \emph{unlabeled} Cayley graph on a group $G$ is defined as an equivalence class of Cayley graphs of $G$ under the relation of graph isomorphism. Throughout, we shall not distinguish between an equivalence class and any of its representatives. With this convention, we obtain the following unlabeled analogue of Theorem~\ref{thmNumber} and Corollary~\ref{cor}.

For each $\delta\in(0,1/2)$, {since $h_\delta(r)$ approaches $0$ as $r\to\infty$, there exists $r_\delta>0$ such that $h_\delta(r)<1$ whenever $r\geq r_\delta$. For $r\geq r_\delta$,} let
\begin{equation}\label{eqK}
    k_\delta(r)=\frac{h_\delta(r)}{1-h_\delta(r)}\cdot 2^{(\log_2r)^2+\log_2r},
\end{equation}
where $h_\delta(r)$ is defined in~\eqref{eqHdelta}.
It is clear that $k_\delta(r)$ approaches $0$ as $r\to \infty$.

\begin{theorem}\label{thmIso}
    Let $G$ be an abelian group of order $r$ and exponent greater than $2$, let $\iota$ be the inversion on $G$, and let $\delta\in(0,1/2)$. Then for sufficiently large $r$, the proportion of unlabeled Cayley graphs $\Gamma$ on $G$ with
    \[
    \Aut(D(\Gamma)) = (R(G) \rtimes \langle \iota \rangle) \times C_2
    \]
    is at least $1 - k_\delta(r)$, where $k_\delta(r)$ is defined in~\eqref{eqK}. In particular, for sufficiently large $r$, the proportion of unlabeled Cayley graphs of $G$ that are stable MRRs is at least $1 - k_\delta(r)$.
\end{theorem}

\begin{corollary}\label{corIso1}
    Let $G$ be an abelian group of order $r$, and let $\iota$ be the inversion on $G$. Then the proportion of unlabeled Cayley graphs $\Gamma$ on $G$ with
    \[
    \Aut(D(\Gamma)) = (R(G) \rtimes \langle \iota \rangle) \times C_2
    \]
    approaches $1$ as $r\to \infty$. In particular, the proportion of unlabeled Cayley graphs of $G$ that are stable MRRs approaches $1$ as $r\to \infty$.
\end{corollary}

The paper is organized as follows. In Section~\ref{sec2}, we present several preliminary results. In Section~\ref{sec3}, we reduce the proof of our main theorem, Theorem~\ref{thmNumber}, to four asymptotic estimates. Section~\ref{sec4} is devoted to establishing these estimates. Finally, in Section~\ref{sec5}, we combine the estimates to prove Theorem~\ref{thmNumber}, and derive Theorem~\ref{thmIso} together with the corollaries.

\section{Preliminaries}\label{sec2}
%
%For a group $G$ and $S$ an inverse-closed subset of $G$, the \emph{Cayley graph} of $G$ with \emph{connection set} $S$, denoted by $\Cay(G,S)$, is defined to be the graph with vertex set $G$ such that $x$ and $y$ are adjacent if and only if $yx^{-1}\in S$. Note that we do not exclude the possibility of $S$ containing the identity of $G$, in which case the corresponding Cayley graph has a loop at every vertex.

In this paper, $U-V:=\{u\in U\mid u\notin V\}$ denotes the set difference for sets $U$ and $V$, and $U\sqcup V$ denotes the disjoint union of $U$ and $V$. For a positive integer $n$, denote by $\pi(n)$ the set of prime divisors of $n$, denote by $C_n$ the cyclic group of order $n$, and if $n$ is a power of a prime $p$, then denote by $E_n$ the elementary abelian $p$-group of order $n$. Throughout this paper, we fix $G$ to be an abelian group of order $r$ and of exponent greater than $2$, and fix the inversion mapping
 \[
    \iota\colon\, G\to G,\ \ g\mapsto g^{-1}.
 \]
 For an element $t\in G$, let $R(t)$ be the permutation on $G$ sending each $g\in G$ to $gt$. For a subset $T\subseteq G$, denote $R(T)=\{R(t)\mid t\in T\}$.

 For subgroups $H$ and $K$ of a group $X$, let $X/H$ denote the set of right cosets of $H$ in $X$, and let $H\backslash X/K=\{HxK\mid x\in X\}$ denote the set of double cosets of $H$ and $K$ in $X$. Denote the center of $X$ by $\Center(X)$, the normalizer of $H$ in $X$ by $\Nor_X(H)$, the core $\bigcap_{x\in X}H^x$ of $H$ in $X$ by $\Core_X(H)$, and the order of an element $x\in X$ by $|x|$. For $T\subseteq X$, denote
 \[
    \calI(T)=\{t\in T\mid |t|\leq 2\} \ \text{ and }\ \bfc(T)=\frac{|T|+|\calI(T)|}{2}.
 \]
 Observe that, if $T$ is inverse-closed, then the number of inverse-closed subsets of $T$ is $2^{\bfc(T)}$ (see~\cite[Lemma~2.2]{S2021}, for instance).

  \begin{definition}[{\cite[Definition~2.1]{DX2000}}]\label{defBiCos}
    Let $X$ be a group with subgroups $H$ and $K$, and let $D$ be a union of double cosets of $K$ and $H$ in $X$, namely, $D=\bigcup_i Kd_iH$. Define the \emph{bi-coset graph of $X$ with respect to $H$, $K$ and $D$}, denoted by $\BiCos(X,H,K;D)$, to be the bipartite graph with two parts $X/H$ and $X/K$ such that two vertices $Hx$ and $Ky$ from these two parts, respectively, are adjacent if and only if $yx^{-1}\in D$.
\end{definition}

 Let $\Gamma$ be a graph with vertex set $V$, and let $D(\Gamma)$ be the standard double cover of $\Gamma$, so that $D(\Gamma)$ has vertex set $V\times\mathbb{Z}_2=(V\times \{0\})\cup (V\times \{1\})$. For $v\in V$, denote by $\Gamma(v)$ the neighborhood of $v$ in $\Gamma$. For convenience, for each $v\in V$, we write $v^+$ and $v^-$ for the vertices $(v,0)$ and $(v,1)$ in $D(\Gamma)$, respectively. Moreover, for $T\subseteq V$, let
 \[
    T^+=\{v^+\mid v\in T\} \ \text{ and }\  T^-=\{v^-\mid v\in T\}.
 \]
 In particular, $D(\Gamma)$ is bipartite with parts $V^+$ and $V^-$.

  Let $X$ be a subgroup of $\Aut(D(\Gamma))$ that stabilizes $V^+$ (and hence also $V^-$), and let $\varepsilon\in \{+,-\}$. For each $x\in X$, denote by $x^\varepsilon$ the permutation on $V^\varepsilon$ induced by $x$; that is, the image of the homomorphism $X\to\Sym(V^\varepsilon)$ restricting $x$ to $V^\varepsilon$. Denote
 \[
    X^\varepsilon=\{x^\varepsilon\mid x\in X\}\leq \Sym(V^\varepsilon).
 \]

A graph is said to be \emph{twin-free} (also known as \emph{R-thin}, \emph{worthy} or \emph{vertex-determining}) if no two vertices have the same neighborhood.

 \begin{lemma}\label{lmA}
    Let $\Gamma$ be a graph with vertex set $V$, let $A=\Aut(D(\Gamma))$, and let $B$ be the subgroup of $A$ stabilizing $V^+$. Then the following statements hold.
    \begin{enumerate}[\rm(a)]
        \item \label{enuAa} If $\Gamma$ is connected and non-bipartite, then $D(\Gamma)$ is connected, and $A=B\rtimes C_2$.
        \item \label{enuAb} If $\Gamma$ is twin-free, then $B^+\cong B\cong B^-$.
    \end{enumerate}
\end{lemma}

\begin{proof}
    Suppose that $\Gamma$ is connected and non-bipartite. Then it follows from~\cite[Lemma~3.3(1)]{GLP2004} that $D(\Gamma)$ is connected. Hence, $A$ preserves the bipartition $\{V^+,V^-\}$, and so $|A|/|B|\leq 2$. Since the mapping that interchanges $v^+$ and $v^-$ for each $v\in V$ is an involution in $A-B$, we conclude that $A=B\rtimes C_2$, completing the proof of part~\eqref{enuAa}.

    For part~\eqref{enuAb}, suppose that $\Gamma$ is twin-free, and let $\varepsilon\in\{+,-\}$. Since the restriction $x\mapsto x^\varepsilon$ is an epimorphism from $B$ to $B^\varepsilon$, it suffices to show that, given $x^\varepsilon=1$, we have $x=1$. Since $x^\varepsilon=1$, the neighborhood in $D(\Gamma)$ of each vertex in $V^{-\varepsilon}$ is stabilized by $x$. Then since $\Gamma$ is twin-free, it follows that $x$ fixes $V^{-\varepsilon}$ pointwise, which together with $x^\varepsilon=1$ gives that $x=1$. This completes the proof.
\end{proof}

  By abuse of notation, for each $\alpha\in\Sym(G)$, we also write $\alpha$ for the permutation on $G\times \mathbb{Z}_2$ defined by
 \begin{equation}\label{eqAbuse}
    (g,i)\mapsto (g^\alpha,i) \ \text{ for}\  g\in G \ \text{and}\ i\in\mathbb{Z}_2.
 \end{equation}
 Similarly, for each subgroup $H\leq \Sym(G)$, we also use $H$ to denote the corresponding subgroup of $\Sym(G\times \mathbb{Z}_2)$ induced by the subgroup $H$ of $\Sym(G)$. In particular, $R(G)$ can also mean the semiregular subgroup of $\Sym(G\times \mathbb{Z}_2)$ with orbits $G^+$ and $G^-$. Throughout Subsections~\ref{subsec41} and~\ref{subsec43}, $R(G)$ is understood in this sense, while elsewhere its intended meaning is clear from the context. Moreover, for $\varepsilon\in \{+,-\}$ and $H\leq \Sym(G)$, the mapping $h\mapsto h^\varepsilon$ from $H$ to $H^\varepsilon$ is a group isomorphism, where $H^\varepsilon$ is interpreted as the restriction of $H\leq \Sym(G\times\mathbb{Z}_2)$ to $G^\varepsilon$.

\begin{lemma}\label{lmBiCos}
    Let $\Gamma=D(\Sigma)$ for some graph $\Sigma$ with vertex set $V$, let $X$ be a subgroup of $\Aut(\Gamma)$ with orbits $V^+$ and $V^-$, let $H=X_{v^+}$ and $K=X_{v^-}$ for some $v\in V$, let
    \[
        Y=\{x\in X\mid (v^-)^x\in \Gamma(v^+)\},
    \]
    and let $\varphi\colon V^+\cup V^-\to (X/H)\sqcup (X/K)$ mapping $(v^+)^x$ to $Hx$ and $(v^-)^x$ to $Kx$ for each $x\in X$. Then the following statements hold.
    \begin{enumerate}[\rm(a)]
        \item \label{enu22a} $Y$ is a union of double cosets of $K$ and $H$ in $X$.
        \item \label{enu22b} $\varphi$ is a graph isomorphism from $\Gamma$ to $\BiCos(X,H,K;Y)$.
        \item \label{enu22c} Let $\Gamma'=D(\Sigma')$ for some graph $\Sigma'$ with vertex set $V$ such that $X\leq \Aut(\Gamma')$ and $\Gamma^\varphi=(\Gamma')^\varphi$. Then $\Sigma=\Sigma'$.
        \item \label{enu22d} Suppose that $\Sigma=\Cay(G,S)$ for some inverse-closed $S\subseteq G$, that $v=1$, and that $X$ contains $R(G)$. Then for each $g\in G$, it holds that $KR(g)H\subseteq Y$ if and only if $KR(g)^{-1}H \subseteq Y$.
    \end{enumerate}
\end{lemma}

\begin{proof}
    Parts~\eqref{enu22a} and~\eqref{enu22b} follow from~\cite[Lemma~2.4]{DX2000} and its proof. For part~\eqref{enu22c}, let
    \[
        Y'=\{x\in X\mid (v^-)^x\in \Gamma'(v^+)\}.
    \]
    Then $(\Gamma')^\varphi=\BiCos(X,H,K;Y')$ by part~\eqref{enu22b}. Note that $Y$ is the union of the neighbors (as right cosets of $K$) of the vertex $H$, and similarly for $Y'$. Then we deduce from $\Gamma^\varphi=(\Gamma')^\varphi$ that $Y=Y'$, and then $\Gamma(v^+)=\Gamma'(v^+)$ by the definition of $Y$ and $Y'$. Since $X$ is transitive on $V^+$, it follows for each $u^+\in V^+$ that $\Gamma(u^+)=\Gamma'(u^+)$. Thus, by the definition of standard double cover, we conclude that $\Sigma=\Sigma'$, completing the proof of part~\eqref{enu22c}.

    Now we embark on part~\eqref{enu22d}. Take an arbitrary $g\in G$. We deduce from part~\eqref{enu22a} that $KR(g)H\subseteq Y$ if and only if $R(g)\in Y$. Moreover, the definition of $Y$ indicates that $R(g)\in Y$ if and only if $g^-\in \Gamma(1^+)=S^-$, namely, $g\in S$. Therefore, $KR(g)H\subseteq Y$ if and only if $g\in S$. Similarly, $KR(g)^{-1}H\subseteq Y$ if and only if $g^{-1}\in S$. Hence, part~\eqref{enu22d} holds as $S=S^{-1}$.
\end{proof}

\begin{definition}\label{defEquiv}
    For $i\in \{1,2\}$, let $R_i\leq X_i\leq \Sym(\Omega_i)$ for some finite set $\Omega_i$. The triples $(\Omega_1,X_1,R_1)$ and $(\Omega_2,X_2,R_2)$ are \emph{equivalent} if there exists a bijection $\varphi\colon\Omega_1\to\Omega_2$ such that
    \[
        X_2=\varphi^{-1}X_1\varphi
        \quad\text{and}\quad
        R_2=\varphi^{-1}R_1\varphi.
    \]
\end{definition}

It is easy to see that Definition~\ref{defEquiv} defines an equivalence relation.

\begin{definition}
    For a group $X$ and subgroups $Y$ and $T$, we call $(X/Y,X,T)$ a \emph{right triple} if $X$ and $T$ act faithfully on $X/Y$ by right multiplication.
\end{definition}

For a group $X$, let $\Hol(X)$ denote the holomorph of $X$. The following lemma establishes several properties related to the equivalence of triples.

\begin{lemma}\label{lmTriple}
    Let $\Omega$ be a finite set, let $\omega\in\Omega$, and let $X$ be a subgroup of $\Sym(\Omega)$ containing a transitive subgroup $T$. Then the following statements hold.
    \begin{enumerate}[\rm(a)]
        \item \label{eunTriplea} $(\Omega,X,T)$ is equivalent to the right triple $(X/X_\omega,X,T)$.
        \item \label{eunTripleb} For each group isomorphism $\psi\colon X\to \tilde{X}$, the right triple $(X/X_\omega,X,T)$ is equivalent to the right triple $(X^\psi/(X_\omega)^\psi,X^\psi,T^\psi)$.
        \item \label{eunTriplec} Suppose that $T$ is regular, and that there are precisely $\ell$ conjugacy classes of subgroups in $X$ isomorphic to $X_\omega$ and $m$ conjugacy classes of subgroups in $X$ isomorphic to $T$. Then
        \[
            \{(\Omega,Y,T)\mid T\leq Y\leq \Sym(\Omega),\ Y\cong X,\ Y_\omega\cong X_\omega\}
        \]
        can be partitioned into at most $\ell m$ equivalence classes.
        \item \label{eunTripled} Suppose that the set $\Lambda:=\{(G,Y,R(G))\mid R(G)\leq Y\leq \Sym(G)\}$ is contained in a single equivalence class. Then $|\Lambda|\leq |\Hol(G)|$.
    \end{enumerate}
\end{lemma}

\begin{proof}
    \eqref{eunTriplea} Since $X$ contains a transitive subgroup $T$, it is direct to verify that the bijection
    \[
        \varphi\colon\, \Omega\to X/X_\omega,\ \ \omega^x\mapsto X_\omega x
    \]
    provides the required equivalence.

    \eqref{eunTripleb} The isomorphism $\psi$ naturally induces a bijection
    \[
        X/X_\omega \to X^\psi/(X_\omega)^\psi, \ \ X_\omega x\mapsto (X_\omega)^\psi x^\psi.
    \]
    It is routine to verify that this bijection gives the required equivalence.

    \eqref{eunTriplec} Let $i\in \{1,2\}$, and let $Y_i$ be a subgroup of $\Sym(\Omega)$ containing $T$ with $Y_i\cong X$ and $(Y_i)_\omega\cong X_\omega$. Let $\psi_i\colon Y_i\to X$ be a group isomorphism, and let $H_i=((Y_i)_\omega)^{\psi_i}$ and $K_i=T^{\psi_i}$. Then $H_i\cong (Y_i)_\omega\cong X_\omega$ and $K_i\cong T$. Moreover, since $T$ is regular, we have $Y_i=(Y_i)_\omega T$, and so $X=Y_i^{\psi_i}=H_iK_i$. It follows from part~\eqref{eunTriplea} that $(\Omega,Y_i,T)$ is equivalent to the right triple $(Y_i/(Y_i)_\omega,Y_i,T)$, which, by part~\eqref{eunTripleb}, is equivalent to the right triple
    \[
        \big((Y_i)^{\psi_i}/((Y_i)_\omega)^{\psi_i},(Y_i)^{\psi_i},T^{\psi_i} \big)=\big(X/H_i,X,K_i\big).
    \]
    Suppose that there exist $x,y\in X$ such that
    \[
        H_1=H_2^x \ \text{ and }\  K_1=K_2^y.
    \]
    It suffices to show that the right triples $(X/H_1,X,K_1)$ and $(X/H_2,X,K_2)$ are equivalent.

    Noting that $H_1$ and {$H_2=H_1^{x^{-1}}$} are point-stabilizers in $X$ of two (possibly the same) points under the right multiplication action on $X/H_1$, we derive from part~\eqref{eunTriplea} that the right triples $(X/H_1,X,K_1)$ and $(X/H_2,X,K_1)$ are equivalent. Since $X=K_2H_2$, there exists $z\in H_2$ such that $K_2^z=K_2^y=K_1$. Thus, part~\eqref{eunTripleb} yields that the right triple $(X/H_2,X,K_1)$ is equivalent to the right triple
    \[
        (X^{z^{-1}}/H_2^{z^{-1}},X^{z^{-1}},K_1^{z^{-1}})=(X/H_2,X,K_2).
    \]
    This completes the proof of part~\eqref{eunTriplec}.

    \eqref{eunTripled} Fix some $(G,Y_1,R(G))\in \Lambda$, and take an arbitrary $(G,Y_2,R(G))\in \Lambda$. Since $(G,Y_1,R(G))$ and $(G,Y_2,R(G))$ are equivalent, there exists $\varphi\in\Sym(G)$ such that $Y_2=Y_1^\varphi$ and $R(G)=R(G)^\varphi$. Hence, $\varphi\in \Nor_{\Sym(G)}(R(G))=\Hol(G)$, and so there are at most $|\Hol(G)|$ choices for $\varphi$. Consequently, there are at most $|\Hol(G)|$ choices for $(G,Y_2,R(G))$.
\end{proof}

For a set $\Omega$ and a permutation $\tau\in \Sym(\Omega)$, let $\Fix_\Omega(\tau)$ denote the set of elements of $\Omega$ fixed by $\tau$. Note that each element $\alpha$ in $\Hol(G)=\Nor_{\Sym(G)}(R(G))$ has the form $R(g)\tau$ for some $g\in G$ and $\tau\in\Aut(G)$. While the following lemma is stated for our abelian group $G$, it actually remains valid for any finite group.

\begin{lemma}\label{lmfix}
    Let $\alpha=R(g)\tau$ for some $g\in G$ and $\tau\in\Aut(G)$. Then $\Fix_G(\alpha)$ is either empty or a right coset of the subgroup $\Fix_G(\tau)$ in $G$.
\end{lemma}

\begin{proof}
    Suppose $\Fix_G(\alpha)\neq \varnothing$. Fix some $h\in \Fix_G(\alpha)$. Then $h=h^\alpha=h^{R(g)\tau}=(hg)^\tau$, and so $(h^\tau)^{-1}h=g^\tau$. Similarly, for each $k\in\Fix_G(\alpha)$, it holds that $(k^\tau)^{-1}k=g^\tau$, and hence,
    \[
        (h^\tau)^{-1}h=g^\tau=(k^\tau)^{-1}k.
    \]
    This implies $(kh^{-1})^{\tau}=kh^{-1}$, namely, $kh^{-1}\in\Fix_G(\tau)$. Thus, $\Fix_G(\alpha)=\Fix_G(\tau)h$ is a right coset of $\Fix_G(\tau)$ in $G$.
\end{proof}

We end this section with a fact that is used repeatedly in the rest of the paper without reference. Any group $X$ of order $n$ can be generated by at most $\lfloor \log_2n\rfloor$ elements. Hence, the number of subgroups of $X$ is at most $n^{\log_2n}=2^{(\log_2n)^2}$, and similarly, $|\Aut(X)|\leq 2^{(\log_2n)^2}$.

\section{Reduction of the proof}\label{sec3}

In this section, we reduce the proof of our main theorem to four asymptotic estimates.
We first identify a family of connection sets that give rise to stable Cayley graphs, and then decompose its complement into several more tractable ones, each of which will be shown to be small in the next section.

For each inverse-closed subset $S$ of $G$, denote by
\[
    B(S)=\Aut(D(\Cay(G,S)))_{G^+}
\]
the subgroup of $\Aut(D(\Cay(G,S)))$ stabilizing $G^+$.  Then $B(S)$ contains $R(G)$ and $\iota$, where $\iota$ is regarded as a permutation on $G\times\mathbb{Z}_2$ as in~\eqref{eqAbuse}. Since $G$ is abelian, it is direct to verify that $\iota^{-1}R(g)\iota=R(g)^{-1}=R(g^\iota)$ for every $g\in G$. Hence
 \[
    R(G)\rtimes \la\iota\ra\leq B(S).
 \]
The notation $B(S)$ will be used throughout the rest of the paper. We now define the families of connection sets used in the reduction:
\begin{align*}
    \calS   &= \{S\subseteq G \mid  S=S^{-1}\},\\
    \calS_1 &=\{S\in \calS \mid  \Cay(G,S) \text{ is connected, non-bipartite, and twin-free}\},\\
    \calS_2 &=\{S\in \calS_1 \mid  B(S)=\Nor_{B(S)}(R(G))=R(G)\rtimes \la \iota\ra\},\\
    \calS_3 &=\{S\in \calS_1 \mid  R(G)\rtimes \la \iota\ra<\Nor_{B(S)}(R(G))\},\\
    \calS_4 &=\{S\in \calS_1 \mid  \text{there exists } X\leq B(S)  \text{ with } R(G)  \text{ maximal in } X \text{ and } \Nor_X(R(G))=R(G) \},\\
    \calS_5 &= \{S\in \calS_1 \mid \text{there exists } X\leq B(S)  \text{ such that } R(G)\rtimes\langle \iota\rangle = \Nor_X(R(G))<X  \text{ and }  \\
          &\hspace{5.6em}\Nor_X(R(G)) \text{ is the unique subgroup } Y \text{ of } X \text{ satisfying } R(G)<Y<X\}.
\end{align*}

The next lemma shows that every connection set in $\calS_2$ gives rise to a stable Cayley graph.

\begin{lemma}\label{lmStable}
    For each $S\in \calS_2$, we have
    \[
    \Aut(D(\Cay(G,S)))=(R(G)\rtimes\la\iota\ra)\times C_2,
    \]
    and so $\Cay(G,S)$ is stable.
\end{lemma}

\begin{proof}
    Take an arbitrary $S\in \calS_2$, and let $A(S)=\Aut(D(\Cay(G,S)))$ and $\Gamma=\Cay(G,S)$. Clearly, $R(G)\rtimes \la\iota\ra\leq \Aut(\Gamma)$, and so
    \[
        |R(G)\rtimes \la\iota\ra|\leq |\Aut(\Gamma)|.
    \]
    Noting that $\Aut(\Gamma)$ can be viewed as a subgroup of $B(S)$, we obtain
    \[
        |\Aut(\Gamma)|\leq |B(S)|=|R(G)\rtimes \la\iota\ra|
    \]
    as $S\in\calS_2$. Therefore, $\Aut(\Gamma)=B(S)=R(G)\rtimes\la\iota\ra$. Since $S\in\calS_1$, it follows from Lemma~\ref{lmA}\eqref{enuAa} that $|A(S)|=2|B(S)|$, and thus $|A(S)|=2|\Aut(\Gamma)|$. Hence, the fact $A(S)\geq \Aut(\Gamma)\times C_2$ implies that $A(S)=\Aut(\Gamma)\times C_2=(R(G)\rtimes\la\iota\ra)\times C_2$, which completes the proof.
\end{proof}

Thus, to prove Theorem~\ref{thmNumber}, it suffices to establish a suitable upper bound on the proportion $|\calS-\calS_2|/|\calS|$. We shall use the following containment.

\begin{proposition}\label{prop45}
    $(\calS_1-\calS_2)-\calS_3 \subseteq \calS_4\cup \calS_5$.
\end{proposition}

\begin{proof}
For convenience, write $H \maxsub K$ to mean that $H$ is a maximal subgroup of the group $K$. Suppose that $S\in (\calS_1-\calS_2)-\calS_3$. Then $R(G)\rtimes \la\iota\ra =\Nor_{B(S)}(R(G))<B(S)$. Hence, there exists $X\leq B(S)$ such that $\Nor_{B(S)}(R(G))\maxsub X$. It follows from
\[
    R(G)\rtimes \la\iota\ra\leq \Nor_X(R(G))\leq \Nor_{B(S)}(R(G))=R(G)\rtimes \la \iota\ra
\]
that $\Nor_X(R(G))=R(G)\rtimes\la \iota\ra=\Nor_{B(S)}(R(G))$. Thus,
\[
    R(G) \maxsub \Nor_X(R(G)) \ \text{ and }\  \Nor_X(R(G)) \maxsub X.
\]

Assume that $S\notin\calS_5$. Then there exists $Y\leq X$ such that $R(G)<Y<X$ and $Y\neq \Nor_X(R(G))$. Noting that $R(G)\leq \Nor_Y(R(G))\leq \Nor_X(R(G))$, we derive from $R(G) \maxsub \Nor_X(R(G))$ that
\[
    \Nor_Y(R(G))=R(G) \ \text{ or }\ \Nor_Y(R(G))=\Nor_X(R(G)).
\]
If $\Nor_Y(R(G))=\Nor_X(R(G))$, then it follows from $\Nor_X(R(G))\neq Y$ that $\Nor_Y(R(G))< Y<X$, contradicting that $\Nor_X(R(G))\maxsub X$. Therefore, $\Nor_Y(R(G))=R(G)$. Let $Z$ be a subgroup of $Y$ such that $R(G)\maxsub Z$. Then
\[
    R(G)\leq \Nor_Z(R(G))\leq \Nor_Y(R(G))=R(G).
\]
Hence, $\Nor_Z(R(G))=R(G)$, and thus $S\in\calS_4$. This completes the proof.
\end{proof}

This proposition implies
\[
        \calS-\calS_2\subseteq (\calS-\calS_1)\cup\calS_3\cup\calS_4\cup\calS_5.
\]
Thus, we have reduced the problem to estimating $|\calS-\calS_1|/|\calS|$, $|\calS_3|/|\calS|$, $|\calS_4|/|\calS|$ and $|\calS_5|/|\calS|$.

\section{Bounding the four proportions}\label{sec4}

The purpose of this section is to establish upper bounds for the four proportions above. We shall show that $\calS-\calS_1$, $\calS_3$, $\calS_4$ and $\calS_5$ are all asymptotically negligible in $\calS$.

\subsection{Bounding $|\calS-\calS_1|/|\calS|$}\label{subsec31}

In this subsection, we bound $|\calS-\calS_1|/|\calS|$.

\begin{lemma}\label{lmOrder2}
    For each involution $z$ in $G$, the number of subsets of $G$ stabilized by $\la R(z),\iota\ra $ is {at most} $2^{r/4+|\calI(G)|/2}$.
\end{lemma}

\begin{proof}
    Take an arbitrary $x\in G$, and write $x^{\la R(z)\ra}=\{x,y\}$, where $y=xz$. Since $\la R(z),\iota\ra\cong C_2^2$, the orbit $x^{\la R(z),\iota\ra}$ has length $2$ or $4$. We first show that $|x^{\la R(z),\iota\ra}|=2$ if and only if $x^2=1$ or $x^2=z$. Noting that $|x^{\la R(z),\iota\ra}|=2$ if and only if $\{x,y\}^\iota=\{x,y\}$, we are left to show that $\{x,y\}^\iota=\{x,y\}$ if and only if $x^2=1$ or $x^2=z$. If $x^2=1$, then $x=x^{-1}$ and $y^{-1}=(xz)^{-1}=xz=y$. If $x^2=z$, then $y^{-1}=z^{-1}x^{-1}=zx^{-1}=x^2x^{-1}=x$. Hence, $\{x,y\}^\iota=\{x,y\}$ if $x^2=1$ or $x^2=z$. Conversely, suppose that $\{x,y\}^\iota=\{x,y\}$. Then either $x=x^{-1}$ or $x=y^{-1}=z^{-1}x^{-1}=zx^{-1}$, which implies that either $x^2=1$ or $x^2=z$, as desired.

    Let $\mathcal{O}_1=\{x^{\la R(z),\iota\ra}\mid x^2=1\}$, $\mathcal{O}_2=\{x^{\la R(z),\iota\ra}\mid x^2=z\}$ and $\mathcal{O}_3=\{x^{\la R(z),\iota\ra}\mid x^2\notin\{1,z\}\}$. Note that $\calO_1$, $\calO_2$, and $\calO_3$ form a partition of the set of orbits of $\la R(z),\iota\ra$ on $G$, and the conclusion in the above paragraph shows that each orbit in $\calO_1$ and $\calO_2$ has length $2$, while each orbit in $\calO_3$ has length $4$. Since the union of all orbits in $\calO_1$ is $\calI(G)$, it follows that
    \[
        2|\mathcal{O}_1|=|\calI(G)|.
    \]
    {Suppose that $\calO_2\neq\varnothing$. Then there exists some $x\in G$ such that $x^2=z$. Fix such an element $x$. Then for $w\in G$, it holds that $w^2=z$ if and only if $(xw^{-1})^2=1$. Hence, if $\calO_2\neq\varnothing$, then there are $|\calI(G)|$ elements in $G$ whose square equals $z$. Since the union of all orbits in $\calO_2$ is $\{w\in G\mid w^2=z\}$, this implies that $2|\calO_2|=|\calI(G)|$.
    Therefore,
    \[
    |\calO_2|=
    \begin{cases}
    0,&\text{if } \calO_2=\varnothing,\\
    |\calI(G)|/2,&\text{otherwise}.
    \end{cases}
    \]}
    Since each subset of $G$ that is stabilized by $\la R(z),\iota\ra$ is a union of $\la R(z),\iota\ra$-orbits, we conclude that the number of subsets of $G$ that are stabilized by $\la R(z),\iota\ra$ is
    {\[
        2^{|\calO_1|}2^{|\calO_2|}2^{|\calO_3|}=2^{|\calO_1|+|\calO_2|+\frac{r-2|\calO_1|-2|\calO_2|}{4}}=2^{\frac{r+2|\calO_1|+2|\calO_2|}{4}}\leq 2^{\frac{r+|\calI(G)|+|\calI(G)|}{4}}=2^{\frac{r}{4}+\frac{|\calI(G)|}{2}},
    \]}
    as required.
\end{proof}

The following lemma is not only required for the proof of Proposition~\ref{propS1}, but is also of independent interest.

\begin{lemma}\label{lmTrivial}
    The following statements hold.
    \begin{enumerate}[\rm(a)]
        \item \label{enu34a} The proportion of inverse-closed subsets $S$ of $G$ such that $\Cay(G,S)$ is disconnected is at most $2^{-r/4+(\log_2r)^2}$.
        \item \label{enu34b} The proportion of inverse-closed subsets $S$ of $G$ such that $\Cay(G,S)$ is connected and bipartite is at most $2^{-r/4+(\log_2r)^2}$.
        \item \label{enu34c} The proportion of inverse-closed subsets $S$ of $G$ such that $\Cay(G,S)$ is not twin-free is at most $2^{-r/6+\log_2r+1}$.
    \end{enumerate}
\end{lemma}

\begin{proof}
    Recall that the number of inverse-closed subsets of $G$ is
    \[
        2^{\bfc(G)}=2^{\frac{r+|\calI(G)|}{2}}.
    \]
    In the following, we establish upper bounds for the numbers of subsets $S$ in parts~\eqref{enu34a}--\eqref{enu34c}, which give the corresponding upper bounds for the proportions by dividing by $2^{\bfc(G)}$.

    \eqref{enu34a} For $S\subseteq G$, the graph $\Cay(G,S)$ is disconnected if and only if $\la S\ra<G$, which is equivalent to $|\la S\ra|\leq r/2$. We now count the number of inverse-closed subsets $S$ of $G$ with $|\la S\ra|\leq r/2$ in two steps: first count the number of subgroups $H$ of $G$ with $|H|\leq r/2$, and then fix such a subgroup $H$ and count the number of inverse-closed generating sets of $H$. The number of subgroups $H$ of $G$ with $|H|\leq r/2$ is less than the number of subgroups of $G$, which is at most $2^{(\log_2r)^2}$. Now fix a subgroup $H$ of $G$ with $|H|\leq r/2$. Then the number of inverse-closed subsets $S$ of $G$ such that $\la S\ra=H$ is at most
    \[
        2^{\bfc(H)}=2^{\frac{|H|+|\calI(H)|}{2}}\leq 2^{\frac{|H|+|\calI(G)|}{2}}\leq 2^{\frac{r}{4}+\frac{|\calI(G)|}{2}}.
    \]
    Therefore, the number of inverse-closed subsets $S$ of $G$ such that $\Cay(G,S)$ is disconnected is at most
    \[
        2^{(\log_2r)^2}\cdot 2^{\frac{r}{4}+\frac{|\calI(G)|}{2}}=2^{(\log_2r)^2+\frac{r}{4}+\frac{|\calI(G)|}{2}},
    \]
    as desired.

    \eqref{enu34b} If $\Cay(G,S)$ is connected and bipartite, then the part containing $1$ forms a subgroup $H$ of index $2$ in $G$ such that $S\subseteq G- H$. Note that there are at most $2^{(\log_2r)^2}$ subgroups of $G$, and for a fixed subgroup $H$ of $G$ with $|G|/|H|=2$, the number of inverse-closed subsets $S$ of $G- H$ is at most
    \[
        2^{\bfc(G- H)}=2^{\frac{|G- H|}{2}+\frac{|\calI(G- H)|}{2}}\leq 2^{\frac{r}{4}+\frac{|\calI(G)|}{2}}.
    \]
    Consequently, the number of inverse-closed subsets $S$ of $G$ such that $\Cay(G,S)$ is connected and bipartite is at most
    \[
        2^{(\log_2r)^2}\cdot 2^{\frac{r}{4}+\frac{|\calI(G)|}{2}}=2^{(\log_2r)^2+\frac{r}{4}+\frac{|\calI(G)|}{2}},
    \]
    as desired.

    \eqref{enu34c} For each inverse-closed subset $S$ of $G$ such that $\Cay(G,S)$ is not twin-free, we fix a non-identity element $g=g(S)\in G$ (the notation $g(S)$ indicates that the element $g$ depends on $S$) such that $S=Sg$. Hence, $S=S\la g\ra=S^{\la R(g)\ra}$, which together with $S=S^{-1}=S^\iota$ implies that $S$ is stabilized by $\la R(g),\iota\ra$. Let
    \[
        \mathcal{T}=\{S\in\calS\mid \Cay(G,S) \ \text{is not twin-free and}\ |g(S)|=2\}
    \]
    and $\mathcal{T}'=\{g(S)\in G\mid S\in \mathcal{T}\}$. Then
    \[
        \mathcal{T}\subseteq \{S\subseteq G\mid S^{\la R(z),\iota\ra}=S \ \text{for some}\ z\in \mathcal{T}'\}.
    \]
    Hence, since $|\mathcal{T}'|\leq |G|=r$, it follows from Lemma~\ref{lmOrder2} that
    \[
        |\mathcal{T}|\leq r\cdot 2^{\frac{r}{4}+\frac{|\calI(G)|}{2}}= 2^{\frac{r}{4}+\frac{|\calI(G)|}{2}+\log_2r}.
    \]
    To complete the proof, we now define
    \[
        \mathcal{W}=\{S\in\calS\mid \Cay(G,S) \ \text{is not twin-free and}\ |g(S)|\geq 3\}
    \]
    and $\mathcal{W}'=\{g(S)\in G\mid S\in \mathcal{W}\}$. For each $z\in \mathcal{W}'$, since $|z|\geq 3$ and each $S\in \mathcal{W}$ with $g(S)=z$ is a union of $\la R(z)\ra$-orbits, the number of $S\in\mathcal{W}$ such that $g(S)=z$ is at most $2^{r/3}$. Thus,
    \[
        |\mathcal{W}|\leq|\mathcal{W}'|\cdot 2^{\frac{r}{3}}\leq 2^{\frac{r}{3}+\log_2r},
    \]
    as $|\mathcal{W}'|\leq |G|=r$. As a consequence, the number of inverse-closed subsets $S$ of $G$ such that $\Cay(G,S)$ is not twin-free is at most
    \[
        |\mathcal{T}\cup\mathcal{W}|\leq
        2^{\frac{r}{4}+\frac{|\calI(G)|}{2}+\log_2r}+2^{\frac{r}{3}+\log_2r}\leq 2^{\frac{r}{3}+\frac{|\calI(G)|}{2}+\log_2r+1},
    \]
    which completes the proof.
\end{proof}

\begin{remark}
    The conclusions in Lemma~\ref{lmTrivial}\eqref{enu34a}\eqref{enu34b} still hold if $G$ is nonabelian, as can be seen in the proof.
\end{remark}

We now present the main result of this subsection.

\begin{proposition}\label{propS1}
    The proportion $|\calS-\calS_1|/|\calS|$ is at most $2^{-r/6+(\log_2r)^2+2}$.
\end{proposition}

\begin{proof}
    If $S\in \calS- \calS_1$, then $\Cay(G,S)$  is disconnected, or connected and bipartite, or not twin-free. Hence, we conclude from Lemma~\ref{lmTrivial} that
    \[
        \frac{|\calS-\calS_1|}{|\calS|}\leq 2^{-\frac{r}{4}+(\log_2r)^2}+2^{-\frac{r}{4}+(\log_2r)^2}+ 2^{-\frac{r}{6}+\log_2r+1}\leq 2^{-\frac{r}{6}+(\log_2r)^2+2}.
    \]
    This completes the proof.
\end{proof}

\subsection{Bounding $|\calS_3|/|\calS|$}\label{subsec32}

Recall that
\begin{align*}
    \calS_3 &=\{S\in \calS_1 \mid  R(G)\rtimes \la \iota\ra<\Nor_{B(S)}(R(G))\}.
\end{align*}
In this subsection, we give an upper bound for $|\calS_3|/|\calS|$.
This is accomplished by showing $\calS_3\subseteq \calS_3'$  and bounding the proportion of $|\calS_3'|/|\calS|$, where
\[
    \calS_3':=\{S\in\calS_1\mid S^\alpha=S\ \text{for some}\ \alpha\in \Hol(G)- \{1,\iota\}\}.
\]

\begin{lemma}\label{lm35}
    The set $\calS_3$ is contained in $\calS_3'$.
\end{lemma}

\begin{proof}
    Take an arbitrary $S\in\calS_3$. Then there exists a non-identity element
    \[
        \alpha\in \Nor_{B(S)}(R(G))-(R(G)\rtimes\la\iota\ra).
    \]
    Writing $\Gamma=D(\Cay(G,S))$ and $(1^-)^{\alpha}=g^-$ for some $g\in G$, we derive that
    \[
        (S^+)^\alpha=\big(\Gamma(1^-)\big)^\alpha=\Gamma\big((1^-)^\alpha\big)=\Gamma(g^-)=(Sg)^+=(S^+)^{R(g)}.
    \]
    As a consequence,
    \begin{equation}\label{eqSrg}
        (S^+)^{\alpha^+R(g^{-1})^+}=S^+.
    \end{equation}
    Note that $\alpha^+R(g^{-1})^+\notin \{1^+,\iota^+\}$ and
    \[
        \alpha^+R(g^{-1})^+\in\Nor_{B(S)^+}(R(G)^+)\leq \Nor_{\Sym(G)^+}(R(G)^+)=\Hol(G)^+.
    \]
    Therefore, there exists $\beta\in\Hol(G)$ such that $\beta^+=\alpha^+R(g^{-1})^+\notin\{1^+,\iota^+\}$. Therefore, this together with~\eqref{eqSrg} shows that $\beta\in \Hol(G)-\{1,\iota\}$ and $S^\beta=S$. Thus, $S\in\calS_3'$, as desired.
\end{proof}

{\begin{lemma}\label{lmHol}
    The proportion $|\calS_3'|/|\calS|$ is at most $2^{-r/24+(\log_2r)^2+\log_2r}$.
\end{lemma}

\begin{proof}
    For each $S\in\calS_3'$, we fix some $\alpha=\alpha(S)\in\Hol(G)-\{1,\iota\}$ such that $S^\alpha=S$. Since $S=S^{-1}$, the set $S$ is stabilized by $\la\alpha,\iota\ra$. As $\iota$ commutes with every automorphism of $G$, we have $[\alpha,\iota]\in R(G)$. Write $[\alpha,\iota]=R(z)$ for some $z\in G$. Then $Sz=S$. Since $\Cay(G,S)$ is twin-free, this implies $z=1$. Hence $\alpha$ and $\iota$ commute.

    We claim that, for each fixed $\alpha$, the number of sets $S\in\calS_3'$ with $\alpha(S)=\alpha$ is at most $2^{11r/24+|\calI(G)|/2}$. Following the orbit-counting argument in the proof of~\cite[Lemma~5.5]{DSV2016}, it suffices to verify the following statements for every $\beta\in\Hol(G)-\{1,\iota\}$ that commutes with $\iota$.
    \begin{enumerate}[(i)]
        \item \label{enuFixa} $|\Fix_G(\beta)|\leq r/2$, and if $|\Fix_G(\beta)|<r/2$, then $|\Fix_G(\beta)|\leq r/3$.
        \item \label{enuFixb} If $|\calI(G)|<r/2=|\Fix_G(\beta)|=|\Fix_G(\iota\beta)|$, then $|\calI(G)|=r/4$.
    \end{enumerate}
    For~\eqref{enuFixa}, Lemma~\ref{lmfix} shows that the fixed-point set of a non-identity element of $\Hol(G)$ is either empty or a coset of a proper subgroup of $G$. Thus its size is at most $r/2$, and is at most $r/3$ if it is less than $r/2$. For~\eqref{enuFixb}, write $\beta=R(g)\tau$ with $g\in G$ and $\tau\in\Aut(G)$. Since $G$ is abelian, we have $\iota\in\Aut(G)$ and $\iota\beta=R(g^{-1})(\iota\tau)$. Let
    \[
        H=\Fix_G(\tau) \ \text{ and }\ K=\Fix_G(\iota\tau).
    \]
    By the assumption in~\eqref{enuFixb}, both $\Fix_G(\beta)$ and $\Fix_G(\iota\beta)$ are nonempty of size $r/2$. Hence Lemma~\ref{lmfix} shows that they are right cosets of $H$ and of $K$, respectively, which gives $|H|=|K|=r/2$. Each $x\in H\cap K$ satisfies $x=x^{\iota\tau}=(x^\tau)^{-1}=x^{-1}$, and so
    \[
        H\cap K\leq \calI(G).
    \]
    If $H=K$, then $H\leq \calI(G)$ and hence $|\calI(G)|\geq |H|=r/2$, contrary to the assumption $|\calI(G)|<r/2$. Thus $H\neq K$. As $H$ and $K$ both have index $2$ in $G$, this forces $HK=G$, and therefore
    \[
        |H\cap K|=\frac{|H|\,|K|}{|G|}=\frac{r}{4}.
    \]
    Consequently, $\calI(G)$ is a subgroup of $G$ containing the subgroup $H\cap K$ of index $4$, so that $|G:\calI(G)|\in\{1,2,4\}$, that is, $|\calI(G)|\in\{r,r/2,r/4\}$. Since $|\calI(G)|<r/2$, we conclude that $|\calI(G)|=r/4$, proving~\eqref{enuFixb}. Hence the claim follows from the argument in~\cite[Lemma~5.5]{DSV2016}.

    There are at most
    \[
        |\Hol(G)|=r|\Aut(G)|\leq 2^{(\log_2r)^2+\log_2r}
    \]
    choices for $\alpha$. Since $|\calS|=2^{r/2+|\calI(G)|/2}$, we conclude that
    \[
        \frac{|\calS_3'|}{|\calS|}
        \leq \frac{2^{(\log_2r)^2+\log_2r}\cdot 2^{\frac{11}{24}r+\frac{|\calI(G)|}{2}}}{2^{\frac{r}{2}+\frac{|\calI(G)|}{2}}}
        =2^{-\frac{r}{24}+(\log_2r)^2+\log_2r}.
    \]
    This completes the proof.
\end{proof}}

Combining Lemmas~\ref{lm35} and~\ref{lmHol}, we obtain the
following main result of this subsection.

\begin{proposition}\label{propS3}
    The proportion $|\calS_3|/|\calS|$ is at most {$2^{-r/24+(\log_2r)^2+\log_2r}$}.
\end{proposition}

\subsection{Bounding $|\calS_4|/|\calS|$}\label{subsec41}

Recall that
\begin{align*}
    \calS_4 &=\{S\in \calS_1 \mid  \text{there exists } X\leq B(S)  \text{ with } R(G)  \text{ maximal in } X \text{ and } \Nor_X(R(G))=R(G) \}.
\end{align*}
In this subsection, we give an upper bound for $|\calS_4|/|\calS|$.

For each $S\in \calS_4$, fix a subgroup $X(S)$ of $B(S)$ such that $R(G)$ is a maximal subgroup of $X(S)$ and $\Nor_{X(S)}(R(G))=R(G)$. Fix some $\delta \in (0,1/2)$, and let
\[
    \mathcal{T}_\delta=\{S\in \calS_4\mid |\Core_{X(S)}(R(G))|\geq r^{1-\delta} \} \ \text{ and }\ \mathcal{W}_\delta=\calS_4-\mathcal{T}_\delta.
\]
To bound $|\mathcal{S}_4|$, we bound $|\mathcal{T}_\delta|$ and $|\mathcal{W}_\delta|$ respectively.

We require the following result, obtained by applying~\cite[Theorem~3.2]{DSV2016} to $X^+$ for $X=X(S)$ with $S\in\calS_4$, although the same statement also holds for $X^-$.

\begin{lemma}\label{lmXplus}
    Let $X=X(S)$ for some $S\in\calS_4$, let $N<G$ with $R(N)=\Core_X(R(G))$, and let $H=X_{1^+}$. Then there exist $p\in\pi(r)$ and nontrivial subgroups $Q$ and $L$ of $R(G)^+$ such that the following statements hold.
    \begin{enumerate}[\rm(a)]
        \item \label{enuX1}$X^+=(H^+\times Q)\rtimes L$, where $H^+$ is an elementary abelian $p$-group.
        \item \label{enuX2} $R(G)^+=Q\times L$.
        \item \label{enuX3} $R(N)^+=\Center(X^+)=Q\times\Cen_L(H^+)$.
        \item \label{enuX4} $P:=H^+\times Q$ is the unique Sylow $p$-subgroup of $X^+$.
        \item \label{enuX5} $T:=\Nor_{X^+}(H^+)=H^+\times R(N)^+$.
        \item \label{enuX7} $R(G)^+/R(N)^+$ is a nontrivial cyclic group of order coprime to $p$.
        \item \label{enuX8} $H^+R(N)^+/R(N)^+$ is {a faithful irreducible}  $(R(G)^+/R(N)^+)$-module.
        \item \label{enuX9} $H^+(H^+)^y=H^+(H^+)^z$ for each $y,z\in X^+- T$.
    \end{enumerate}
\end{lemma}

{Throughout this subsection, let}
\begin{equation}\label{eqcalX}
    \mathcal{X}=\{X(S)\mid S\in \mathcal{T}_\delta\}.
\end{equation}
We first bound $|\mathcal{X}|$ in Lemma~\ref{lmXprime}, and then, for each $X\in\mathcal{X}$, we bound the number of $S\in \mathcal{T}_\delta$ such that $X(S)=X$ in Lemma~\ref{lmX}. These two lemmas result in an upper bound of $|\mathcal{T}_\delta|$.

The following lemma gathers several technical consequences of Lemma~\ref{lmXplus} that will be used later in the proof.

\begin{lemma}\label{lmfurther}
    Let $X\in \mathcal{X}$, let $N<G$ with $R(N)=\Core_X(R(G))$, and let $H=X_{1^+}$. Suppose that, for $p\in\pi(r)$ and for $1<Q\leq R(G)^+$ and $1<L\leq R(G)^+$, Lemma~$\ref{lmXplus}$\eqref{enuX1}--\eqref{enuX9} hold. Then the following statements hold.
    \begin{enumerate}[\rm(a)]
        \item \label{enuHCpi} $H\cong C_p^i$ for some positive integer $i\leq r^\delta$.
        \item \label{enuT} $T$ is a normal subgroup of $X^+$ containing $P$ such that $T\cap R(G)^+=R(N)^+$.
        \item \label{enuHHs} For each $y\in X$, the product $HH^y$ is a subgroup of $X$, and so there exists  $M\leq G$ such that
        \begin{equation}\label{eqRM}
            R(M)=(HH^y)\cap R(G).
        \end{equation}
        \item \label{enuM} Let $y\in R(G-N)$, and let $M\leq G$ be as in~\eqref{eqRM}. Then $1<M\leq N$, and $M$ is independent of the choice of $y$ from $R(G-N)$. Moreover, if $H$ is conjugate to $X_{1^-}$ in $X$, then there exists $g\in G$ such that, for each $h\in G-gN$, the $H$-orbit $(h^-)^H$ is also an $R(M)$-orbit.
    \end{enumerate}
\end{lemma}

\begin{proof}
    Note from Lemma~\ref{lmA}\eqref{enuAb} that $X^+\cong X\cong X^-$, and then from Lemma~\ref{lmXplus}\eqref{enuX1} that $H\cong H^+$ is an elementary abelian $p$-group.

    \eqref{enuHCpi}
    Let $n=|N|$, and let $H\cong C_p^i$ for some positive integer $i$, and we are left to show that $i\leq r^\delta$. Noting that $H^+\cong H^+R(N)^+/R(N)^+$ as $H^+\cap R(N)^+=1$, we derive from Lemma~\ref{lmXplus}\eqref{enuX7}\eqref{enuX8} that $H^+$ is a {faithful} irreducible module of {$R(G)^+/R(N)^+\cong C_{r/n}$, where $r/n$ is coprime to $p$, and hence
    \[
        i=\min\{j\in\mathbb{Z}_{>0}\mid r/n \ \text{divides}\ p^j-1\}\leq|(\mathbb{Z}_{r/n})^\times|<r/n.
    \]}
    Recall from~\eqref{eqcalX} that $n\geq r^{1-\delta}$. We obtain $i< r^\delta$, completing the proof of part~\eqref{enuHCpi}.

    \eqref{enuT} Lemma~\ref{lmXplus}\eqref{enuX3}--\eqref{enuX5} imply that $T=P\Center(X^+)$, which is a product of normal subgroups of $X^+$. Hence, $T$ is normal in $X^+$. Since $R(G)^+$ is regular, we derive from Lemma~\ref{lmXplus}\eqref{enuX5} that
    \[
        T\cap R(G)^+=R(N)^+,
    \]
    which completes the proof of part~\eqref{enuT}.

    \eqref{enuHHs} Since both $H^+$ and $(H^+)^{y^+}$ are $p$-groups, they are contained in $P$. Noting from Lemma~\ref{lmXplus}\eqref{enuX4} that $P$ is abelian, we conclude that $H^+(H^+)^{y^+}$ is a subgroup. Therefore, $HH^y$ is a subgroup of $X$, and so part~\eqref{enuHHs} holds.

    \eqref{enuM}
    By part~\eqref{enuT}, $ T\cap R(G)^+=R(N)^+$. This together with $y\in R(G)-R(N)$ forces that $y^+\notin T$, and so by Lemma~\ref{lmXplus}\eqref{enuX9}, the group $M$ does not depend on the choice of $y$. Since $y^+\notin T=\Nor_{X^+}(H^+)$ by Lemma~\ref{lmXplus}\eqref{enuX5}, we have
    \begin{equation}\label{eqHgt}
        H^+ (H^+)^{y^+}>H^+.
    \end{equation}
    Combining this with~\eqref{eqRM} and $X^+=H^+R(G)^+$, we derive that
    \[
        R(M)^+=(H^+(H^+)^{y^+})\cap R(G)^+>1,
    \]
    and so $M>1$. Since $H^+$ and $(H^+)^{y^+}$ are both $p$-groups,
    \[
        H^+(H^+)^{y^+}\leq P\leq  T
    \]
    by part~\eqref{enuT}. This in conjunction with part~\eqref{enuT} leads to
    \[
        R(M)^+=(H^+(H^+)^{y^+})\cap R(G)^+\leq T\cap R(G)^+=R(N)^+.
    \]
    It follows that $M\leq N$, which completes the proof of $1<M\leq N$.

    Finally, suppose that $H=(X_{1^-})^x$ for some $x\in X$, and let $g^-=(1^-)^x$ for some $g\in G$. Take an arbitrary $h\in G-gN$. It remains to prove that $(h^-)^H$ is an $R(M)$-orbit. Write $h^-=(g^-)^{z}$ for some $z\in R(G)$. Then
    \begin{equation}\label{eqXh}
        H^{z}=(X_{1^-})^{xz}=X_{(1^-)^{xz}}=X_{(g^-)^{z}}=X_{h^-}.
    \end{equation}
    Since $X=X_{h^-}R(G)$, we have $H^{z}H=X_{h^-}H=\bigcup_j X_{h^-}R(g_j)$ for some $g_j\in G$, and thus
    \begin{equation}\label{eqHcapG}
        (h^-)^{H^{z}H}=(h^-)^{(H^{z}H)\cap R(G)}.
    \end{equation}
    If $z^+\in T$, then $z^+\in T\cap R(G)^+=R(N)^+$, and so $z\in R(N)$. In this case, $h^-=(g^-)^{z}\in (gN)^-$, contradicting that $h\in G-gN$. Therefore, $z^+\notin T$. Then Lemma~\ref{lmXplus}\eqref{enuX9} gives that $H^+(H^+)^{y^+}=H^+(H^+)^{z^+}$, and hence $HH^y=HH^{z}$. As a result,
    \begin{equation}\label{eqRMRG}
        R(M)=HH^y\cap R(G)=HH^{z}\cap R(G).
    \end{equation}
    Moreover, it follows from part~\eqref{enuHHs} that $H^{z}H=HH^{z}$. Combining this with~\eqref{eqXh}--\eqref{eqRMRG}, we conclude that
    \[
        (h^-)^H=(h^-)^{H^{z}H}=(h^-)^{(H^{z}H)\cap R(G)}=(h^-)^{(HH^{z})\cap R(G)}=(h^-)^{R(M)},
    \]
    as desired.
\end{proof}

\begin{lemma}\label{lmXprime}
    $|\mathcal{X}|\leq 2^{(r^{2\delta}+r^\delta+4)(\log_2r)^2+(2+\delta)\log_2r}$.
\end{lemma}

\begin{proof}
    For $p\in \pi(r)$, nonnegative integer $i\leq r^\delta$, nontrivial subgroups $Q$ and $L$ of $R(G)^+$, and $N<G$, let
    \begin{align*}
        \mathcal{X}(p,i,Q,L,N)=\{X\in \mathcal{X}\mid &~ X_{1^+}\cong C_p^i,\ R(N)=\Core_X(R(G)),\\
        &~(p,i,Q,L,N,X)\ \text{satisfies}\ \text{Lemma~\ref{lmXplus}\eqref{enuX1}--\eqref{enuX9}}\}.
    \end{align*}
    We see from Lemmas~\ref{lmXplus} and~\ref{lmfurther}\eqref{enuHCpi} that
    \[
        \mathcal{X}=\bigcup_{p\in \pi(r),\ 1\leq i\leq r^\delta,\ 1<Q\leq R(G)^+,\ 1<L\leq R(G)^+,\ N<G}\mathcal{X}(p,i,Q,L,N).
    \]
    The number of choices for $p$ is at most $r$, the number of choices for $i$ is at most $r^\delta$, and the number of choices for $(Q,L,N)$ is at most
    \[
        2^{(\log_2r)^2}\cdot 2^{(\log_2r)^2}\cdot 2^{(\log_2r)^2}=2^{3(\log_2r)^2}.
    \]
    Therefore, it suffices to show that, fixing $p\in\pi(r)$, positive integer $i\leq r^\delta$, nontrivial subgroups $Q$ and $L$ of $R(G)^+$, and proper normal subgroup $N$ of $G$, we have
    \[
        |\mathcal{X}(p,i,Q,L,N)|\leq 2^{(r^{2\delta}+r^\delta+1)(\log_2r)^2+\log_2r}.
    \]
    Take an arbitrary $X\in \mathcal{X}(p,i,Q,L,N)$, and denote $H=X_{1^+}$. Then $R(N)=\Core_X(R(G))$, $H^+\cong C_p^i$, and Lemma~\ref{lmXplus}\eqref{enuX1}--\eqref{enuX9} hold. Moreover, $X^+=(H^+\times Q)\rtimes_\theta L$ for some homomorphism $\theta\colon L\to \Aut(H^+\times Q)$ such that $L^\theta$ acts trivially on $Q$.

    We now bound the number $m$ of triples $(G^+,\tilde{X}^+,R(G)^+)$ in
    \[
        \Omega:=\{(G^+,\tilde{X}^+,R(G)^+)\mid \tilde{X}\in \mathcal{X}(p,i,Q,L,N)\}
    \]
    up to equivalence (in the sense of Definition~\ref{defEquiv}). Let $\tilde{X}\in \mathcal{X}(p,i,Q,L,N)$ and let $\tilde{H}=\tilde{X}_{1^+}$. Then $\tilde{H}\cong C_p^i\cong H$ with $i\leq r^\delta$, and
    \[
        X^+=(H^+\times Q)\rtimes_{\theta} L \ \text{ and }\ \tilde{X}^+=(\tilde{H}^+\times Q)\rtimes_{\tilde{\theta}} L
    \]
    for some homomorphisms $\theta\colon L\to \Aut(H\times Q)$ and $\tilde{\theta}\colon L\to \Aut(\tilde{H}\times Q)$ such that both $L^\theta$ and $L^{\tilde{\theta}}$ act trivially on $Q$. Let $\varphi\colon H\times Q \to \tilde{H}\times Q$ be an isomorphism such that
    \[
        H^\varphi=\tilde{H} \ \text{ and }\ q^\varphi=q \ \text{ for each}\ q\in Q.
    \]
    It is direct to verify that $\varphi$ induces an isomorphism
    \[
        \psi\colon\, (H\times Q)\rtimes_\theta L\to (\tilde{H}\times Q)\rtimes_\sigma L,\ \ x\ell\mapsto x^\varphi\ell \ \text{ for}\ x\in H\times Q \ \text{and}\  \ell\in L,
    \]
    where $\sigma$ is a homomorphism from $L$ to $\Aut(\tilde{H}\times Q)$ defined by
    \[
       \ell^\sigma=\varphi^{-1}\ell^\theta \varphi \ \text{  for}\ \ell\in L.
    \]
    In particular,
    \[
        H^\psi=\tilde{H} \ \text{ and }\ (R(G)^+)^\psi=(Q\times L)^\psi=Q\times L= R(G)^+.
    \]
    Moreover, if $\sigma=\tilde{\theta}$, then $(X^+)^\psi=\tilde{X}^+$. Note by Lemma~\ref{lmTriple}\eqref{eunTriplea} that $(G^+,X^+,R(G))$ is equivalent to the right triple $(X^+/H,X^+,R(G)^+)$, and that $(G^+,\tilde{X}^+,R(G)^+)$ is equivalent to the right triple $(\tilde{X}^+/\tilde{H},\tilde{X}^+,R(G)^+)$. Furthermore, Lemma~\ref{lmTriple}\eqref{eunTripleb} asserts that the right triple $(X^+/H,X^+,R(G)^+)$ is equivalent to the right triple
    \[
        ((X^+)^\psi/H^\psi, (X^+)^\psi, (R(G)^+)^\psi)=((X^+)^\psi/\tilde{H}, (X^+)^\psi, R(G)^+).
    \]
    Hence, if $\sigma=\tilde{\theta}$, then $(G^+,X^+,R(G))$ is equivalent to $(G^+,\tilde{X}^+,R(G))$. Therefore, we infer that $m$ is at most the cardinality of the set
    \[
        \Theta:=\{\rho \mid \rho \ \text{is a homomorphism from}\ L \ \text{to}\ \Aut(\tilde{H}\times Q), \  L^\rho \ \text{acts trivially on}\ Q\}.
    \]
    For fixed generators $\ell_1,\ldots,\ell_k$ of $L$, where $k\leq \log_2|L|\leq \log_2r$, each $\rho\in \Theta$ is determined by $\ell_1^\rho,\ldots,\ell_k^\rho$, and each $\ell_j^\rho$ is determined by its image on an ordered set of $i$ generators of $\tilde{H}$. Hence,
    \[
        |\Theta|\leq |\tilde{H}\times Q|^{i\log_2r}=(p^i|Q|)^{i\log_2r}\leq (r^{r^\delta}r)^{r^\delta\log_2r}=2^{r^\delta(r^\delta+1)(\log_2r)^2}.
    \]
    Therefore, $m\leq 2^{r^\delta(r^\delta+1)(\log_2r)^2}$.

    By Lemma~\ref{lmTriple}\eqref{eunTripled}, the intersection of each equivalence class (in the sense of Definition~\ref{defEquiv}) with $\Omega$ has size at most $|\Hol(R(G)^+)|\leq 2^{(\log_2r)^2+\log_2r}$. Therefore,
    \[
        |\mathcal{X}(p,i,Q,L,N)|=|\Omega|\leq 2^{r^\delta(r^\delta+1)(\log_2r)^2}\cdot 2^{(\log_2r)^2+\log_2r}=2^{(r^{2\delta}+r^\delta+1)(\log_2r)^2+\log_2r},
    \]
    as desired.
\end{proof}

For a subgroup $M$ of $G$, since $G$ is abelian, right cosets of $M$ coincide with left cosets of $M$, which will be simply referred to as $M$-cosets.

\begin{lemma}\label{lmConj}
    Let $g\in G$ and $1<M\leq N<G$. Suppose that $G$ is not a $2$-group. Then the number of $S\in\calS$ such that $S-gN$ is a union of $M$-cosets is at most $2^{11r/24+|\calI(G)|/2}$.
\end{lemma}

\begin{proof}
    Let $n=|N|$ and $m=|M|$. Since $G$ is abelian, we see that
    \begin{equation}\label{eqiff}
        h^2\in M \ \text{ if and only if }\ hM=(hM)^{-1}.
    \end{equation}

    Assume first that $gN\neq g^{-1}N$. Then $g^{-1}N\subseteq G-gN$. Let $\ell=|\{h\in G\mid h^2\in M\}|$, and let
    \[
        k=|\{hM\in G/M\mid hM=(hM)^{-1}\}|.
    \]
    Then $k=\ell/m$ by~\eqref{eqiff}. We note from $M\leq N$ that $g^{-1}N$ is a union of $M$-cosets. Take an arbitrary $S\in\calS$ such that $S-gN$ is a union of $M$-cosets. Then
    \[
        S\cap g^{-1}N=(S\cap g^{-1}N)\cap(G-gN)=(S-gN)\cap g^{-1}N
    \]
    is a union of $M$-cosets. This combined with $S=S^{-1}$ implies that $S\cap gN=(S\cap g^{-1}N)^{-1}$ is a union of $M$-cosets, and so $S=(S\cap gN)\cup (S-gN)$ is a union of $M$-cosets. Hence, the number of $S\in\calS$ such that $S-gN$ is a union of $M$-cosets is at most
    \[
        2^{k+\frac{1}{2}\left(\frac{r}{m}-k\right)}
        =2^{\frac{\ell}{m}+\frac{1}{2}\left(\frac{r}{m}-\frac{\ell}{m}\right)}
        =2^{\frac{r+\ell}{2m}}.
    \]
    Thus, it suffices to show that
    \[
        \frac{r+\ell}{2m}\leq \frac{11}{24}r+ \frac{|\calI(G)|}{2}.
    \]
    If $\ell=r$, then $g^2\in M\leq N$, contradicting $gN\neq g^{-1}N$. Therefore, $\ell\leq r/2$, and so
    \[
        \frac{r+\ell}{2m}\leq \frac{r+\frac{r}{2}}{4}=\frac{3}{8}r<\frac{11}{24}r+\frac{|\calI(G)|}{2},
    \]
    as desired.

    Assume next that $gN=g^{-1}N$. Let $\ell=|\{h\in G-gN\mid h^2\in M\}|$. By~\eqref{eqiff}, the number of $S\in\calS$ such that $S-gN$ is a union of $M$-cosets is
    \[
        2^{|\calI(gN)|+\frac{n-|\calI(gN)|}{2}+\frac{\ell}{m}+\frac{r-n-\ell}{2m}}=2^{\frac{|\calI(gN)|}{2}+\frac{n}{2}(1-\frac{1}{m})+\frac{r+\ell}{2m}}\leq 2^{\frac{|\calI(G)|}{2}+\frac{n}{2}(1-\frac{1}{m})+\frac{r+\ell}{2m}}.
    \]
    Hence, it suffices to show that
    \[
        \frac{n}{2}\Big(1-\frac{1}{m}\Big)+\frac{r+\ell}{2m}\leq \frac{11}{24}r.
    \]
    If $m=2$, then elements of $G$ whose square lies in $M$ is contained in a Sylow $2$-subgroup of $G$, which implies $\ell\leq r/3$ as $G$ is not a $2$-group. In this case, we derive from $n\leq r/2$ that
    \[
        \frac{n}{2}\Big(1-\frac{1}{m}\Big)+\frac{r+\ell}{2m}\leq \frac{r}{4}\Big(1-\frac{1}{2}\Big)+\frac{r+\frac{r}{3}}{4}=\frac{11}{24}r,
    \]
    as desired. If $m\geq 3$, then since $\ell\leq r-n$ and $n\leq r/2$, we deduce that
    \begin{align*}
        \frac{n}{2}\Big(1-\frac{1}{m}\Big)+\frac{r+\ell}{2m}
        &\leq \frac{n}{2}\Big(1-\frac{1}{m}\Big)+\frac{2r-n}{2m}\\
        &=\frac{n}{2}\Big(1-\frac{2}{m}\Big)+\frac{r}{m}
        \leq \frac{r}{4}\Big(1-\frac{2}{m}\Big)+\frac{r}{m}=\frac{r}{4}+\frac{r}{2m}
        \leq \frac{r}{4}+\frac{r}{6}<\frac{11}{24}r,
    \end{align*}
    completing the proof.
\end{proof}

\begin{lemma}\label{lmX}
    For each $X\in \mathcal{X}$, the number of $S\in \mathcal{T}_\delta$ such that $X(S)=X$ is at most $2^{|\calI(G)|/2+11r/24+2(\log_2 r)^2}$.
\end{lemma}

\begin{proof}
    Recall from Lemma~\ref{lmA}\eqref{enuAb} that $X^+\cong X\cong X^-$. Let $H=X_{1^+}$, let $K=X_{1^-}$, and let $N< G$ with $R(N)=\Core_X(R(G))$. Since $X\in\mathcal{X}$, there exist $p\in\pi(r)$ and nontrivial subgroups $Q$ and $L$ of $R(G)^+$ such that Lemma~\ref{lmXplus}\eqref{enuX1}--\eqref{enuX9} hold. In particular, $G$ is not a $2$-group.

    \textsf{Case 1}: $H$ and $K$ are conjugate in $X$. We obtain from Lemma~\ref{lmfurther}\eqref{enuM} that there exist $M\leq N$ and $g\in G$ such that $M>1$ and, for each $h\in G-gN$, the $H$-orbit $(h^-)^H$ is an $R(M)$-orbit. Note from $H=X_{1^+}$ that $H$ fixes $S^-$ setwise. Hence, $(S-gN)^-$ is a union of $R(M)$-orbits, and so $S-gN$ is a union of $M$-cosets. There are at most $2^{2(\log_2r)^2}$ choices for $(M,N)$, and Lemma~\ref{lmConj} shows that for a fixed $(M,N)$, the number of $S\in\mathcal{T}_{\delta}$ such that $S-gN$ is a union of $M$-cosets is at most $2^{11r/24+|\calI(G)|/2}$. Therefore, the conclusion of the lemma follows.

    \textsf{Case 2}: $H$ and $K$ are not conjugate in $X$. Then
    \begin{equation}\label{eqDoubleCoset}
        \frac{|KxH|}{|H|}=\frac{|H|}{|K^x\cap H|}>1 \ \text{ for each}\ x\in X.
    \end{equation}
    By Lemma~\ref{lmBiCos}\eqref{enu22a}--\eqref{enu22c}, distinct choices of $S\in\mathcal{T}_\delta$ with $X(S)=X$ correspond to distinct choices of $Y$ such that $Y$ is a union of double cosets in $K\backslash X/H$ with
    \[
        D(\Cay(G,S))^\varphi= {\BiCos(X,H,K;Y)},
    \]
    where $\varphi$ is defined in Lemma~\ref{lmBiCos}. Since $X=R(G)H$, each double coset in $K\backslash X/H$ has the form $KR(g)H$ for some $g\in G$. Let
    \[
        \kappa=|K\backslash X/H| \ \text{ and }\ \kappa'=|\{KR(g)H\mid g\in G,\ KR(g)^{-1}H=KR(g)H\}|.
    \]
    Noting from Lemma~\ref{lmBiCos}\eqref{enu22d} that $KR(g)H\subseteq Y$ if and only if $KR(g)^{-1}H\subseteq Y$, the number of $S\in \mathcal{T}_\delta$ such that $X(S)=X$ is at most $2^{\kappa'+(\kappa-\kappa')/2}=2^{(\kappa+\kappa')/2}$. Thus, we are left to show that
    \[
        \frac{\kappa+\kappa'}{2}\leq \frac{|\calI(G)|}{2}+\frac{11}{24}r+2(\log_2 r)^2.
    \]

    \textsf{Subcase 2.1}: $p\geq 3$. Since Lemma~\ref{lmXplus}\eqref{enuX1}\eqref{enuX2} together with $X^+\cong X$ show that $|H|=|X|/r$ is a power of $p$, we observe from~\eqref{eqDoubleCoset} that $|KxH|/|H|\geq 3$ for each $x\in X$, and so $3\kappa\leq |X|/|H|=r$. Therefore,
    \[
        \frac{\kappa+\kappa'}{2}\leq \kappa\leq \frac{r}{3}< \frac{|\calI(G)|}{2}+\frac{11}{24}r+2(\log_2 r)^2,
    \]
    as desired.

    \textsf{Subcase 2.2}: $p=2$. Take an arbitrary $g\in G$ such that $KR(g)^{-1}H=KR(g)H$. We claim that $g\in N$. Write
    \begin{equation}\label{eqRgkh}
        R(g)^{-1}=kR(g)h
    \end{equation}
    for some $k\in K$ and $h\in H$. Let $T=\Nor_{X^+}(H^+)$ be as in Lemma~\ref{lmXplus}\eqref{enuX5}. Then $T$ is normal in $X^+$ as Lemma~\ref{lmfurther}\eqref{enuT} asserts. Since $|H|=|X|/|R(G)|=|K|$, it follows from Lemma~\ref{lmXplus}\eqref{enuX1} that both $H$ and $K$ are $p$-subgroups of $X$. Then we conclude from Lemma~\ref{lmfurther}\eqref{enuT} that
    \begin{equation}\label{eqHKT}
        H^+\leq T \ \text{ and }\ K^+\leq T.
    \end{equation}
    For each $x\in X$, denote by $\overline{x}$ the image of the natural homomorphism $X^+\to X^+/T$. Then we deduce from~\eqref{eqRgkh} and~\eqref{eqHKT} that
    \[
        \overline{R(g)^+}^{-1}=\overline{k^+}\overline{R(g)^+}\overline{h^+}=\overline{R(g)^+}.
    \]
    Since Lemma~\ref{lmfurther}\eqref{enuT} implies that $X^+/T$ is a $2'$-group, we obtain $\overline{R(g)^+}=\overline{1^+}$. This together with Lemma~\ref{lmfurther}\eqref{enuT} yields that $R(g)^+\in T\cap R(G)^+=R(N)^+$, and so $g\in N$, as claimed.

    From the above claim we derive that $\kappa'\leq |N|$. Since Lemma~\ref{lmXplus}\eqref{enuX7} shows that $R(G)^+/R(N)^+$ is a nontrivial $2'$-group, it follows that $\kappa'\leq|N|\leq r/3$. As a consequence, recalling from~\eqref{eqDoubleCoset} and Lemma~\ref{lmXplus}\eqref{enuX1}\eqref{enuX2} that $2\kappa\leq |X|/|H|=r$, we conclude that
    \[
        \frac{\kappa+\kappa'}{2}\leq \frac{\frac{r}{2}+\frac{r}{3}}{2}=\frac{5}{12}r<\frac{|\calI(G)|}{2}+\frac{11}{24}r+2(\log_2 r)^2.
    \]
    This completes the proof.
\end{proof}

We are now ready to establish the main result of this subsection.

\begin{proposition}\label{propT}
    $|\mathcal{T}_\delta|\leq 2^{\bfc(G)-r/24+(r^{2\delta}+r^\delta+6)(\log_2r)^2+(2+\delta)\log_2r}$.
\end{proposition}

\begin{proof}
    We obtain from Lemma~\ref{lmXprime} that $|\mathcal{X}|\leq 2^{(r^{2\delta}+r^\delta+4)(\log_2r)^2+(2+\delta)\log_2r}$ and from Lemma~\ref{lmX} that, for each $X\in \mathcal{X}$, the number of $S\in \mathcal{T}_\delta$ with $X(S)=X$ is at most $2^{|\calI(G)|/2+11r/24+2(\log_2 r)^2}$. Hence, recalling that $\bfc(G)=(r+|\calI(G)|)/2$, we derive
    \begin{align*}
        |\mathcal{T}_\delta|&\leq 2^{(r^{2\delta}+r^\delta+4)(\log_2r)^2+(2+\delta)\log_2r}\cdot 2^{\frac{|\calI(G)|}{2}+\frac{11}{24}r+2(\log_2 r)^2}\\
        &=2^{\bfc(G)-\frac{r}{24}+(r^{2\delta}+r^\delta+6)(\log_2r)^2+(2+\delta)\log_2r}.
    \end{align*}
    This completes the proof.
\end{proof}

The following is devoted to bounding $|\mathcal{W}_\delta|$. The argument is largely based on the method of Babai and Godsil~\cite[Section~4]{BG1982}, further developed by Morris, Moscatiello and Spiga~\cite[Section~4]{MMS2022}, where we reproduce certain parts with minor modifications adapted to our context.

Recall that $G$ is an abelian group of order $r$. Let $N$ be a subgroup of $G$ of order $n>1$ such that $G/N\cong C_b$ for some integer $b\geq 2$, and label coset representatives of $N$ in $G$ as $\gamma_0,\ldots,\gamma_{b-1}$ such that $N\gamma_iN\gamma_j=N\gamma_{i+j}$, where subscripts of $\gamma$ are counted modulo $b$. Set $\calO_i=N\gamma_i$ for each $i\in \mathbb{Z}_b$. Then $\calO_i$s are orbits of $R(N)$ on $G$ with $|\calO_i|=|N|$. For $S\in \calS$, $u\in G$ and $j\in \mathbb{Z}_b$, let
\begin{equation}\label{eqSigmaSuj}
    \sigma(S,u,j)=S\cap Su\cap \calO_j.
\end{equation}
If $u\in \calO_i=N\gamma_i$, then $\gamma_i^{-1}u\in N$, and so $\calO_{j-i}u=N\gamma_j\gamma_i^{-1}u=N\gamma_j=\calO_j$. In this case,
\begin{align}
    \sigma(S,u,j)&=(S\cap \calO_j)\cap (Su\cap \calO_j)\nonumber\\
    &=(S\cap \calO_j)\cap ((S\cap \calO_ju^{-1})u)=(S\cap \calO_j)\cap ((S\cap \calO_{j-i})u).\label{eqSigma}
\end{align}
For distinct $u,v\in \calO_i$, let
\[
    \Psi_i(\{u,v\},j)=\{S\in \calS \mid |\sigma(S,u,j)|=|\sigma(S,v,j)|\}
\]
for each $j\in \mathbb{Z}_b$, and let
\[
    \Psi_i(\{u,v\})=\bigcap_{j\in \mathbb{Z}_b-\{0,i\}} \Psi_i(\{u,v\},j).
\]

\begin{lemma}\label{lmPsi}
    For each $i\in \mathbb{Z}_b-\{0\}$ and $u,v\in\calO_i$ with $u\neq v$, we have
    \[
        |\Psi_i(\{u,v\})|
        \leq
        2^{\bfc(G)-\frac{b}{25}+1}.
    \]
\end{lemma}

\begin{proof}
    Define an auxiliary graph $\Gamma$ with vertex set $\{\{j,-j\}\mid j\in \mathbb{Z}_b\}$ such that
    \begin{enumerate}[\rm(i)]
        \item if $2i=0$, then for each $j\in\mathbb{Z}_b$, the vertex $\{j,-j\}$ is adjacent to
        \[
            \{j-i,-j+i\} \ \text{ and }\ \{j+i,-j-i\};
        \]
        \item if $2i\neq 0$, then for each $j\in\mathbb{Z}_b$, the vertex $\{j,-j\}$ is adjacent to
        \[
            \{j-i,-j+i\},\ \{j+i,-j-i\},\  \{j-2i,-j+2i\}, \ \text{and}\  \{j+2i,-j-2i\}.
        \]
    \end{enumerate}
    Obviously, each vertex of $\Gamma$ has valency at most $4$. Let
    \[
        \Lambda=\{j\in \mathbb{Z}_b\mid 2j\notin\{0,\pm i,\pm 2i\}\}.
    \]
    Then $|\Lambda|\geq b-10$, and no vertex $\{j,-j\}$ with $j\in\Lambda$ has a loop.     {Moreover, for each $j\in\Lambda$, we have
    $j\notin\{0,i\}$ and $2j\notin\{0,i,2i\}$. Hence cases~(2)--(5) of
    \cite[Proposition~4.2]{MMS2022} cannot occur, and so}
    \begin{equation}\label{eqPsi}
        |\Psi_i(\{u,v\},j)|\leq \frac{3}{4}\cdot2^{\bfc(G)} \ \text{ for each}\ j\in \Lambda.
    \end{equation}
    Let $\Gamma'$ be the induced subgraph of $\Gamma$ on $\{\{j,-j\}\mid j\in \Lambda\}$. For each $j\in \Lambda$, we derive from $2j\neq 0$  that $j\neq -j$,  and so the graph $\Gamma'$ has $|\Lambda|/2$ vertices. By Caro--Tur\'{a}n--Wei~\cite{C1979,T1941,W1981}, $\Gamma'$ has an independent set of cardinality at least
    \[
    \sum_{x\in V(\Gamma')}
    \frac{1}{\deg_{\Gamma'}(x)+1}
    \geq
    \frac{|V(\Gamma')|}{5}
    \geq
    \frac{b-10}{10}.
    \]

    Let $\{\{j,-j\}\mid j\in \Delta\}$ be such an independent set, where $\Delta\subseteq \Lambda$  with $|\Delta|\geq (b-10)/10$. {Here $\Delta$ contains exactly one
    representative from each vertex of the independent set.} Then for distinct $j,k\in \Delta$,
    \[
        \{\pm j,\pm(j-i),\pm(j+i)\}\cap \{\pm k,\pm(k-i),\pm(k+i)\}=\varnothing.
    \]
        {By~\eqref{eqSigma}, for each $j\in\Delta$, membership in
    $\Psi_i(\{u,v\},j)$ is determined by
    \[
        S\cap
        \big(\calO_j\cup\calO_{-j}
        \cup\calO_{j-i}\cup\calO_{-j+i}\big).
    \]
    These inverse-closed sets are pairwise disjoint as $j$ ranges
    over $\Delta$. Since their inverse-closed subsets can be chosen
    independently, the events $\Psi_i(\{u,v\},j)$, $j\in\Delta$,
    are mutually independent when $S$ is chosen uniformly from
    $\calS$.} Hence, by~\eqref{eqPsi},

    \[
        |\Psi_i(\{u,v\})|\leq \Big(\frac{3}{4}\Big)^{|\Delta|}\cdot2^{\bfc(G)}\leq\Big(\frac34\Big)^{\frac{b-10}{10}}2^{\bfc(G)}=2^{\bfc(G)-\frac{\log_2(4/3)}{10}(b-10)}\leq 2^{\bfc(G)-\frac{b}{25}+1}.
    \]

    This completes the proof.
\end{proof}

\begin{remark}
    The proof of Lemma~\ref{lmPsi} imitates that of~\cite[Proposition~4.4]{MMS2022}. However, there is a minor error in their proof: the independence of their set $\calI$ does not guarantee that the neighborhoods of $\{j_u,j_u^{-1}\}$ and $\{j_v,j_v^{-1}\}$ are disjoint. This issue can be resolved by slightly modifying the definition of the auxiliary graph, in a manner similar to our construction above. Such a modification changes some coefficients in the statement of~\cite[Proposition~4.4]{MMS2022}, but does not affect their main result.
\end{remark}

We are now ready to prove the main result of this subsection.

\begin{proposition}\label{propW}
    $|\mathcal{W}_\delta|\leq  2^{\bfc(G)-\frac{r^\delta}{25}+(\log_2 r)^2+3\log_2r+1}$.
\end{proposition}

\begin{proof}
Take an arbitrary $S\in \mathcal{W}_\delta$, let $X=X(S)$, and let $N\leq G$ such that $R(N)=\Core_X(R(G))$. Note that $N$ depends on $S$. Recall from Lemma~\ref{lmA}\eqref{enuAb} that $X^+\cong X$. Let $n=|N|$ and $b=r/n$. Noting from Lemma~\ref{lmXplus}\eqref{enuX3}\eqref{enuX7} that $n>1$ and $G/N$ is cyclic of order $b>1$, we adopt the notation above Lemma~\ref{lmPsi}. Then $\{\calO_0^+,\calO_1^+,\ldots,\calO_{b-1}^+\}$ is an $X$-invariant partition of $G^+$. Since $S\in\calS_4$, it follows that
\begin{equation}\label{eqMaxNor}
    R(G)^+ \ \text{is maximal in}\ X^+, \ \text{and}\ \Nor_{X^+}(R(G)^+)=R(G)^+.
\end{equation}
Then $R(G)^+/R(N)^+$ is core-free and maximal in $X^+/R(N)^+$, and $\Nor_{X^+/R(N)^+}(R(G)^+/R(N)^+)=R(G)^+/R(N)^+$. Let $K$ be the subgroup of $X$ stabilizing each orbit of $R(N)$ on $G^+\cup G^-$. Then $X/K$ is a permutation group on $(G/N)^+\cup (G/N)^-$.

% BEGIN REV20260913b faithful-blocks
{We claim that $X^+/K^+$ is a permutation group on $(G/N)^+$, that is, the kernel of
the action of $X^+$ on $\{\calO_0^+,\ldots,\calO_{b-1}^+\}$ is precisely $K^+$.
Let $P$ be a Sylow $p$-subgroup of $X$. Since $x\mapsto x^+$ is an isomorphism from $X$ to
$X^+$ by Lemma~\ref{lmA}\eqref{enuAb}, the group $P^+$ is a Sylow $p$-subgroup of $X^+$, which is
unique by Lemma~\ref{lmXplus}\eqref{enuX4}. Hence $P$ is the unique Sylow $p$-subgroup of $X$,
and so $P\trianglelefteq X$. Moreover, $X_{1^+}$ is a $p$-group by Lemma~\ref{lmXplus}\eqref{enuX1},
whence $X_{1^+}\leq P$, and Lemma~\ref{lmXplus}\eqref{enuX7} gives $p\nmid b$ as
$|R(G)^+/R(N)^+|=b$. Fix $\varepsilon\in\{+,-\}$, and let $K_\varepsilon$ be the kernel of
the action of $X$ on $\{\calO_0^\varepsilon,\ldots,\calO_{b-1}^\varepsilon\}$. As $X$ is
transitive on these $b$ blocks and $P$ is a normal $p$-subgroup of $X$, all $P$-orbits on the
blocks have the same length, which is a power of $p$ dividing $b$. Hence this length is $1$,
and so $X_{1^+}\leq P\leq K_\varepsilon$. It follows that
\[
    K_\varepsilon=K_\varepsilon\cap X=K_\varepsilon\cap X_{1^+}R(G)=X_{1^+}(K_\varepsilon\cap R(G))=X_{1^+}R(N).
\]
Thus $K_+=K_-=X_{1^+}R(N)$, and so $K=K_+\cap K_-=K_+$, which proves the claim.}
% END REV20260913b faithful-blocks

Since $R(N)^+\leq K^+$ is transitive on $N^+$, we have $K^+=(K^+)_{1^+}R(N)^+$. Suppose $(K^+)_{1^+}=1$. Then $K^+=R(N)^+$, and so $R(G)^+/R(N)^+=R(G)^+/K^+$ is a regular subgroup of $X^+/R(N)^+=X^+/K^+$. Hence, applying~\cite[Theorem~3.2]{DSV2016} to $X^+/R(N)^+$ and $R(G)^+/R(N)^+$, we obtain that $R(G)^+/R(N)^+$ is not core-free in $X^+/R(N)^+$, a contradiction. Thus, $(K^+)_{1^+}\neq 1$.

We claim that there exists $f\in K$ such that
\begin{equation}\label{eqF}
    f^+\in (K^+)_{1^+}, \ \text{and}\ f|_{\calO_i^+}\neq 1\ \text{for some}\ i\in\mathbb{Z}_b-\{0\}.
\end{equation}
{In fact, since $(K^+)_{1^+}\neq 1$, there exists $g\in K$
such that $1\neq g^+\in(K^+)_{1^+}$. Thus,
$g|_{\calO_j^+}\neq 1$ for some $j\in\mathbb{Z}_b$.} If $g|_{\calO_i^+}\neq 1$ for some $i\in\mathbb{Z}_b-\{0\}$, then we are done. Otherwise, $g|_{\calO_i^+}=1$ for every $i\in\mathbb{Z}_b-\{0\}$, and so $j=0$. Noting that $X^+/K^+$ is transitive on $\{\calO_0^+,\ldots,\calO_{b-1}^+\}$, there exists $x\in X$ such that $(\calO_1^+)^x=\calO_0^+$. It follows that $(\calO_0^+)^x=\calO_k^+$ for some $k\in\mathbb{Z}_b-\{0\}$. Then we derive from $g|_{\calO_1^+}=1$ and $g|_{\calO_0^+}\neq 1$ that $g^x|_{\calO_0^+}=(x^{-1}gx)|_{\calO_0^+}=1$ and $g^x|_{\calO_k^+}=(x^{-1}gx)|_{\calO_k^+}\neq 1$. Moreover, since $1^+\in\calO_0^+$ and $K\trianglelefteq X$, it follows that $(1^+)^{g^x}=1^+$ and $g^x\in K$. Therefore, $f:=g^x$ satisfies~\eqref{eqF}.

We now let $f\in K$ satisfy~\eqref{eqF}. Then $(1^+)^f=1^+$, {$(\calO_j^\varepsilon)^f=\calO_j^\varepsilon$
for each $j\in\mathbb{Z}_b$ and $\varepsilon\in\{+,-\}$}, and there exists $i\in \mathbb{Z}_b-\{0\}$ and distinct $u,v\in \calO_i$ such that $(u^+)^f=v^+$.

As $f\in \Aut(D(\Cay(G,S)))$, we obtain $(S^-)^f=S^-$ and $((Su)^-)^f=(Sv)^-$. Thus, for each $j\in\mathbb{Z}_b$,
\[
    ((S\cap Su\cap \calO_j)^-)^f=(S^-\cap (Su)^-\cap\calO_j^-)^f=S^-\cap (Sv)^-\cap \calO_j^-=(S\cap Sv\cap \calO_j)^-.
\]
This combined with~\eqref{eqSigmaSuj} implies that $|\sigma(S,u,j)|=|\sigma(S,v,j)|$. {Hence $S\in\Psi_i(\{u,v\})$.}

{Let
\[
    \mathcal N=\{N\leq G\mid R(N)=\Core_{X(S)}(R(G))
    \text{ for some }S\in\mathcal W_\delta\}.
\]
For each $N\in\mathcal N$, let $n_N=|N|$ and $b_N=r/n_N$.
Fix a cyclic indexing of the $N$-cosets, and write
$\mathcal O_{N,i}$ and $\Psi_{N,i}(\{u,v\})$ for the corresponding
cosets and events defined before Lemma~\ref{lmPsi}.
By Lemma~\ref{lmXplus} and the definition of $\mathcal W_\delta$,
$n_N>1$, $G/N$ is cyclic, and $b_N>r^\delta$.
The preceding argument gives
\[
    \mathcal W_\delta\subseteq
    \bigcup_{N\in\mathcal N}\ \bigcup_{i\in\mathbb Z_{b_N}-\{0\}}
    \ \bigcup_{\substack{\{u,v\}\subseteq\mathcal O_{N,i}\\u\ne v}}
    \Psi_{N,i}(\{u,v\}).
\]
For each $N\in\mathcal N$ and $i\in\mathbb Z_{b_N}-\{0\}$,
there are at most $\binom{n_N}{2}$ choices for $\{u,v\}$.
Hence Lemma~\ref{lmPsi} yields
\[
    |\mathcal W_\delta|
    \leq\sum_{N\in\mathcal N}(b_N-1)\binom{n_N}{2}
    2^{\bfc(G)-b_N/25+1}.
\]
Since $|\mathcal N|\leq2^{(\log_2r)^2}$ and, for each $N\in\mathcal N$,
$(b_N-1)\binom{n_N}{2}<r^3$ and $b_N>r^\delta$, we obtain
\[
    |\mathcal W_\delta|\leq2^{(\log_2r)^2}r^3 2^{\bfc(G)-r^\delta/25+1}
    =2^{\bfc(G)-r^\delta/25+(\log_2r)^2+3\log_2r+1}.
\]
This completes the proof.{\qedhere}}
\end{proof}

Combining Propositions~\ref{propT} and~\ref{propW}, we derive that
\begin{align*}
    \frac{|\calS_4|}{|\calS|}&=\frac{|\mathcal{T}_\delta|+|\mathcal{W}_\delta|}{2^{\bfc(G)}}\\
    &\leq \frac{2^{\bfc(G)-\frac{r}{24}+(r^{2\delta}+r^\delta+6)(\log_2r)^2+(2+\delta)\log_2r}+{2^{\bfc(G)-\frac{r^\delta}{25}+(\log_2r)^2+3\log_2r+1}}}{2^{\bfc(G)}}\\
    &=2^{-\frac{r}{24}+(r^{2\delta}+r^\delta+6)(\log_2r)^2+(2+\delta)\log_2r}+{2^{-\frac{r^\delta}{25}+(\log_2r)^2+3\log_2r+1}},
\end{align*}
as stated in the following proposition.

\begin{proposition}\label{propS4}
The proportion $|\calS_4|/|\calS|$ is at most
\[
    2^{-\frac{r}{24}+(r^{2\delta}+r^\delta+6)(\log_2r)^2+(2+\delta)\log_2r}+{2^{-\frac{r^\delta}{25}+(\log_2r)^2+3\log_2r+1}}.
\]
\end{proposition}

\subsection{Bounding $|\calS_5|/|\calS|$}\label{subsec43}
Recall that
\begin{align*}
    \calS_5 &= \{S\in \calS_1 \mid \text{there exists } X\leq B(S)  \text{ such that } R(G)\rtimes\langle \iota\rangle = \Nor_X(R(G))<X  \text{ and }  \\
          &\hspace{5.6em}\Nor_X(R(G)) \text{ is the unique subgroup } Y \text{ of } X \text{ satisfying } R(G)<Y<X\}.
\end{align*}
In this subsection, we give an upper bound for $|\calS_5|/|\calS|$.

For each $S\in \calS_5$, fix a subgroup $X(S)$ of $B(S)$ such that $R(G)\rtimes \la\iota\ra=\Nor_{X(S)}(R(G))\leq X(S)$ and $\Nor_{X(S)}(R(G))$ is the unique group with the property that $R(G)<\Nor_{X(S)}(R(G))<X(S)$. {Throughout this subsection, let}
\[
    \mathcal{X}=\{X(S)\mid S\in \calS_5\}.
\]
For each $X\in \mathcal{X}$, applying~\cite[Theorem~3.3]{DSV2016} and the classification of subgroups of $\PGL_2(q)$ (see~\cite[Sections~II.7 and~II.8]{H1967}, for example) to $X^+$ and $X^-$, we obtain the following lemma.

\begin{lemma}\label{lm33}
    Let $\varepsilon\in\{+,-\}$. For each $X\in \mathcal{X}$, there exists a prime power $q\geq 3$ such that the following statements hold.
    \begin{enumerate}[\rm(a)]
        \item \label{enu33a} $X^\varepsilon=U\times \Center(X^\varepsilon)$ with $(X^\varepsilon)_{1^\varepsilon}\leq U\cong\PGL_2(q)$ and $\Center(X^\varepsilon)\cong C_2^\ell$ for some integer $\ell\geq 0$.
        \item \label{enu33b} $R(G)^\varepsilon=C\times \Center(X^\varepsilon)$ with $C\cong C_{q+1}$.
        \item \label{enu33c} {$(X^\varepsilon)_{1^\varepsilon} \cong E_q\rtimes C_{q-1}$ and $(X^\varepsilon)_{1^\varepsilon}\cap U' \cong E_q\rtimes C_{(q-1)/\gcd(2,q-1)}$.}
    \end{enumerate}
\end{lemma}

The upper bound on $|\calS_5|$ is obtained in two steps: first, we give a bound on $|\mathcal{X}|$ in Lemma~\ref{lmS5X}; subsequently, for each fixed $X\in\mathcal{X}$, we bound the number of $S\in\calS_5$ with $X(S)=X$ in Lemma~\ref{lm3.3DoubleCoset}.

\begin{lemma}\label{lmS5X}
    {$|\mathcal{X}|\leq 2^{38(\log_2r)^2+9\log_2r}$.}
\end{lemma}

\begin{proof}
{Let $\mathcal{Q}$ be the set of prime powers $q$ with $3\leq q\leq r-1$.
Given $X\in\mathcal{X}$, for each $\varepsilon\in\{+,-\}$, choose a prime power $q_\varepsilon$ as in Lemma~\ref{lm33}. By Lemma~\ref{lm33}\eqref{enu33a}\eqref{enu33b},
\[
    |X^\varepsilon|
    =q_\varepsilon(q_\varepsilon-1)|R(G)^\varepsilon|
    =q_\varepsilon(q_\varepsilon-1)r.
\]
Since $X^+\cong X^-$, it follows that $q_+=q_-$.
For each $X\in\mathcal{X}$, choose and fix parameters $q_\varepsilon$ as above, and denote their common value by $q(X)$. It is clear that $q(X)\in \mathcal{Q}$.

For $q\in \mathcal{Q}$, let $\mathcal{X}_q=\{X\in\mathcal{X}\mid q(X)=q\}$. Hence we have
\[
    |\mathcal{X}|\leq\sum_{q\in \mathcal{Q}}|\mathcal{X}_q|.
\]
It remains to obtain a uniform upper bound for $|\mathcal{X}_q|$ for those $q\in \mathcal{Q}$ with $\mathcal{X}_q\neq\varnothing$.

Fix $q\in \mathcal{Q}$ with $\mathcal{X}_q\neq\varnothing$. For
$\varepsilon\in\{+,-\}$, let
\[
    \Delta_q^\varepsilon=\{(G^\varepsilon,X^\varepsilon,R(G)^\varepsilon)\mid X\in\mathcal{X}_q\}.
\]
Each $X\in\mathcal{X}_q$ determines a pair
\[
    \big(
    (G^+,X^+,R(G)^+),
    (G^-,X^-,R(G)^-)
    \big)
    \in\Delta_q^+\times\Delta_q^-.
\]
To bound $|\mathcal{X}_q|$, we bound the size of $\Delta_q^+\times\Delta_q^-$ and the maximum number of groups in $\mathcal{X}_q$ giving rise to a given pair.

We first bound $|\Delta_q^+\times\Delta_q^-|$. Fix $\varepsilon\in\{+,-\}$. By Lemma~\ref{lm33}, the groups $X^\varepsilon$, with $X\in\mathcal{X}_q$, are all isomorphic. Fix one such group $X^\varepsilon$. To bound $|\Delta_q^\varepsilon|$ via Lemma~\ref{lmTriple}\eqref{eunTriplec}\eqref{eunTripled}, it suffices to bound the numbers of conjugacy classes in $X^\varepsilon$ of subgroups isomorphic to $(X^\varepsilon)_{1^\varepsilon}$ and to $R(G)^\varepsilon$. It is clear that each of these two numbers is bounded by
\[
    2^{(\log_2|X^\varepsilon|)^2}\leq2^{9(\log_2r)^2}.
\]
Then according to Lemma~\ref{lmTriple}\eqref{eunTriplec}, the triples in $\Delta_q^\varepsilon$ can be partitioned into at most
\[
    2^{18(\log_2r)^2}
\]
equivalence classes. Combining this with Lemma~\ref{lmTriple}\eqref{eunTripled}, we have
\[
    |\Delta_q^\varepsilon|\leq2^{18(\log_2r)^2}|\Hol(G)|\leq2^{19(\log_2r)^2+\log_2r},
\]
and so
\begin{equation}\label{eqDeltaPair}
    |\Delta_q^+\times\Delta_q^-|\leq2^{38(\log_2r)^2+2\log_2r}.
\end{equation}

Next fix $Y\in\mathcal{X}_q$ and consider the groups
$X\in\mathcal{X}_q$ such that
\[
    X^+=Y^+
    \quad\text{and}\quad
    X^-=Y^-.
\]
Let $N=R(G)\rtimes\la \iota\ra $. For each such $X$, choose $x\in X- N$. The definition of $\calS_5$ implies
\[
    X=\la N,x\ra .
\]
Note that the restriction map
\[
    \Sym(G\times\mathbb{Z}_2)_{G^+}
    \longrightarrow
    \Sym(G^+)\times\Sym(G^-),
    \qquad
    \alpha\longmapsto(\alpha^+,\alpha^-),
\]
is injective. Hence $x$ is determined by $(x^+,x^-)\in Y^+\times Y^-$. Since $N$ is fixed and $X=\la N,x\ra $, the number of such groups $X$ is at most $|Y^+||Y^-|\leq r^6$. This together with~\eqref{eqDeltaPair} yields
\[
    |\mathcal{X}_q|\leq2^{38(\log_2r)^2+8\log_2r}.
\]

Finally,  since $|\mathcal{Q}|\leq r$, we conclude that
\[
|\mathcal{X}|\leq \sum_{q\in \mathcal{Q}}|\mathcal{X}_q|\leq r\cdot2^{38(\log_2r)^2+8\log_2r}\leq 2^{38(\log_2r)^2+9\log_2r},
\]
completing the proof.}
\end{proof}

\begin{lemma}\label{lm3.3DoubleCoset}
    For each $X\in \mathcal{X}$, the following statements hold.
    \begin{enumerate}[\rm(a)]
        \item \label{enuZXa} $|\Center(X)|\leq 3r/20+|\calI(G)|/4$.
        \item \label{enuZXb} The number of $S\in\calS_5$ such that $X(S)=X$ is at most $2^{2|\Center(X)|}$.
    \end{enumerate}
\end{lemma}

\begin{proof}
    By Lemma~\ref{lmA}\eqref{enuAb} we have $X^+\cong X\cong X^-$.

    \eqref{enuZXa} It follows from Lemma~\ref{lm33}\eqref{enu33a}\eqref{enu33b} that $R(G)^+= C\times \Center(X^+)$, where $C\cong C_{q+1}$ for some prime power $q\geq 3$ and $\Center(X^+)$ is an elementary abelian $2$-group. Assume first that $q$ is odd. Then we obtain from the structure of $R(G)^+$ that
    \[
        |\calI(R(G)^+)|=2|\Center(X^+)|=\frac{2|R(G)^+|}{|C|}=\frac{2r}{q+1}.
    \]
    Observing $|\calI(G)|=|\calI(R(G)^+)|$ and $|\Center(X)|=|\Center(X^+)|$, we derive that
    \[
        |\Center(X)|-\frac{|\calI(G)|}{4}=\frac{r}{q+1}-\frac{r}{2(q+1)}=\frac{r}{2(q+1)}\leq \frac{r}{8}< \frac{3r}{20}.
    \]
    Assume next that $q$ is even. Then $q\geq 4$, and
    \[
        |\calI(R(G)^+)|=|\Center(X^+)|=\frac{|R(G)^+|}{|C|}=\frac{r}{q+1},
    \]
    which yields that
    \[
        |\Center(X)|-\frac{|\calI(G)|}{4}
        =|\Center(X^+)|-\frac{|\calI(R(G)^+)|}{4}
        =\frac{r}{q+1}-\frac{r}{4(q+1)}=\frac{3r}{4(q+1)}\leq \frac{3r}{20}.
    \]
    This completes the proof of part~\eqref{enuZXa}.

    \eqref{enuZXb} Applying Lemma~\ref{lm33}\eqref{enu33a} with $\varepsilon=+$ and $\varepsilon=-$, respectively, {and noting that $X^+\cong X^-$,} there exists a prime power $q\geq 3$, and subgroups $V$ and $W$ of $X$ isomorphic to $\PGL_2(q)$ such that
    \[
        V\times \Center(X)=X=W\times \Center(X) \ \text{ with }\ X_{1^+}\leq V \ \text{and}\ X_{1^-}\leq W.
    \]
    Clearly, $V'=W'=X'\cong \PSL_2(q)$. Let $d=\gcd(2,q-1)$, and let
    \[
        H=X_{1^+}\cap V' \ \text{ and }\  K=X_{1^-}\cap V'=X_{1^-}\cap W'.
    \]
    We deduce from $X^+\cong X\cong X^-$ and Lemma~\ref{lm33}\eqref{enu33c} that
    \[
        H \cong (X^+)_{1^+}\cap (V')^+\cong {E_q}\rtimes C_{(q-1)/d}\cong (X^-)_{1^-}\cap (V')^-\cong K.
    \]
    Since $\PSL_2(q)$ has a unique conjugacy class of subgroups isomorphic to ${E_q}\rtimes C_{(q-1)/d}$ (see~\cite[Tables~8.1 and~8.2]{BHR2013}), there exists $x\in V'$ such that $K=H^x$.

    With the above information, we bound $|X_{1^+}\backslash X/X_{1^-}|$. {Since $X_{1^+}\cap V'=H$ and $|X_{1^+}:H|=d=|V:V'|$, we have $V=X_{1^+}V'$.} Since $V'$ acts $2$-transitively on $V'/H$ by right multiplication, $|H\backslash V'/H|=2$. Then there exists $y\in V'$ such that
    \[
        V'=H\{1,y\}H=H\{1,y\}K^{x^{-1}}=H\{1,y\}xKx^{-1},
    \]
    and so $V'=H\{x,yx\}K$. {Since $K\leq X_{1^-}$ and $\Center(X)$ is central in $X$, it follows that
    \[
        X=V\times\Center(X)=X_{1^+}V'\Center(X)=X_{1^+}H\{x,yx\}K\Center(X)\subseteq X_{1^+}(\Center(X)\{x,yx\})X_{1^-},
    \]
    and so}
    \[
        |X_{1^+}\backslash X/X_{1^-}|\leq 2|\Center(X)|.
    \]

    Now, Lemma~\ref{lmBiCos}\eqref{enu22a}\eqref{enu22b} shows that each $S\in\calS_5$ with $X(S)=X$ corresponds to a union of double cosets $Y$ in $X_{1^+}\backslash X/X_{1^-}$ such that
    \[
        D(\Cay(G,S))\cong \BiCos(X,X_{1^-},X_{1^+};Y);
    \]
    moreover, Lemma~\ref{lmBiCos}\eqref{enu22c} implies that distinct choices of $S$ correspond to distinct $Y$. As a consequence, the number of $S$ in $\calS_5$ with $X(S)=X$ is at most $2^{|X_{1^+}\backslash X/X_{1^-}|}\leq 2^{2|\Center(X)|}$, as required.
\end{proof}

We are now in a position to give the main result of this subsection.

\begin{proposition}\label{propS5}
The proportion $|\calS_5|/|\calS|$ is at most ${2^{-r/5+38(\log_2r)^2+9\log_2r}}$.
\end{proposition}

\begin{proof}
    Combining statements~\eqref{enuZXa} and~\eqref{enuZXb} of Lemma~\ref{lm3.3DoubleCoset}, we obtain that, for each $X\in \mathcal{X}$, the number of $S$ in $\calS_5$ such that $X(S)=X$ is at most $2^{2|\Center(X)|}\leq 2^{3r/10+|\calI(G)|/2}$. We then conclude from Lemma~\ref{lmS5X} that
\[
    \frac{|\calS_5|}{|\calS|}
    \leq \frac{2^{\frac{3}{10}r+\frac{|\calI(G)|}{2}}|\mathcal{X}|}{|\calS|}
    \leq\frac{2^{\frac{3}{10}r+\frac{|\calI(G)|}{2}+{38(\log_2r)^2+9\log_2r}}}{2^{\frac{r+|\calI(G)|}{2}}}
    =2^{-\frac{r}{5}+{38(\log_2r)^2+9\log_2r}},
\]
which completes the proof.
\end{proof}

\section{Proof of the main results}\label{sec5}

We are now ready to prove the main results.

\subsection{Proof of Theorem~$\ref{thmNumber}$ and its corollaries}\label{subsec44}

In this subsection, we complete the proof of Theorem~\ref{thmNumber} and Corollaries~\ref{cor} and~\ref{corGRR}.

\begin{proof}[Proof of Theorem~$\ref{thmNumber}$]
    It follows from Proposition~\ref{prop45} that
    \[
        |\calS-\calS_2|=|\calS-\calS_1|+|\calS_1-\calS_2|\leq |\calS-\calS_1|+|\calS_3|+|\calS_4|+|\calS_5|.
    \]
    We then conclude from Propositions~\ref{propS1},~\ref{propS3},~\ref{propS4}, and~\ref{propS5} that
    \begin{align*}
        \frac{|\calS-\calS_2|}{|\calS|}&\leq \frac{|\calS-\calS_1|+|\calS_3|+|\calS_4|+|\calS_5|}{|\calS|}\\
        &\leq 2^{-\frac{r}{6}+(\log_2r)^2+2}+{2^{-\frac{r}{24}+(\log_2r)^2+\log_2r}}+2^{-\frac{r}{24}+(r^{2\delta}+r^\delta+6)(\log_2r)^2+(2+\delta)\log_2r}\\
        &\ \ \ {+2^{-\frac{r^\delta}{25}+(\log_2r)^2+3\log_2r+1}+2^{-\frac{r}{5}+38(\log_2r)^2+9\log_2r}}\\
        &\leq {2^{-\frac{r}{24}+(r^{2\delta}+r^\delta+38)(\log_2r)^2+9\log_2r+4}+2^{-\frac{r^\delta}{25}+(\log_2r)^2+3\log_2r+1}}\\
        &=h_\delta(r).
    \end{align*}
    Therefore, the theorem follows from Lemma~\ref{lmStable}.
\end{proof}

\begin{proof}[Proof of Corollary~$\ref{cor}$]
    If $G$ has exponent greater than $2$, then the conclusion follows immediately from Theorem~\ref{thmNumber}. Assume now that the exponent of $G$ is $2$. Then each Cayley digraph on $G$ is a graph, and the inversion mapping $\iota=1$. Hence, it follows from~\cite[Theorem~1.2(b)]{GSX2025} that the proportion of inverse-closed subsets $S\subseteq G$ with
    \begin{equation}\label{eqLong}
        \Aut(D(\Cay(G,S)))=(R(G)\rtimes\la\iota\ra)\times C_2
    \end{equation}
    approaches $1$ as $r\to \infty$. Since $R(G)\rtimes\la\iota\ra\leq \Aut(\Cay(G,S))$, the condition~\eqref{eqLong} implies that $\Cay(G,S)$ is {a stable MRR, indeed a stable GRR}. Consequently, the proportion of inverse-closed subsets $S\subseteq G$ {such that $\Cay(G,S)$ is a stable MRR} approaches $1$ as $r\to \infty$.
\end{proof}

\begin{proof}[Proof of Corollary~$\ref{corGRR}$]
    The corollary follows from Corollary~\ref{cor} and the observations
    \[
        R(G)\rtimes\la\iota\ra \leq \Aut(\Cay(G,S)) \ \text{ and }\  \Aut(\Cay(G,S))\times C_2\leq \Aut(D(\Cay(G,S))).\qedhere
    \]
\end{proof}

\subsection{Proof of Theorem~$\ref{thmIso}$ and its corollary}\label{subsec45}

In this subsection, we prove Theorem~\ref{thmIso} and Corollary~\ref{corIso1}. Note that the condition $\Aut(D(\Cay(G,S)))=(R(G)\rtimes\la\iota\ra)\times C_2$ implies that $\Cay(G,S)$ is stable.

\begin{proof}[Proof of Theorem~$\ref{thmIso}$]
    Let
    \[
        \calS_{\text{good}}=\{S\in\calS\mid \Aut(D(\Cay(G,S)))=(R(G)\rtimes\la\iota\ra)\times C_2\} \ \text{ and }\ \calS_{\text{bad}}=\calS-\calS_{\text{good}}.
    \]
    Then it follows from Theorem~\ref{thmNumber} that $|\calS_{\text{good}}|/|\calS|\geq 1-h_\delta(r)$, where $h_\delta(r)$ is given in~\eqref{eqHdelta}. Since $h_\delta(r)$ approaches $0$ as $r\to \infty$, we know $1-h_\delta(r)>0$ for sufficiently large $r$. Therefore, for sufficiently large $r$, we have
    \begin{equation}\label{eqbad}
        \frac{|\calS_{\text{bad}}|}{|\calS_{\text{good}}|}
        =\frac{|\calS_{\text{bad}}|}{|\calS|}\cdot \frac{|\calS|}{|\calS_{\text{good}}|}
        =\left(1-\frac{|\calS_{\text{good}}|}{|\calS|}\right)\cdot \frac{|\calS|}{|\calS_{\text{good}}|}
        \leq \frac{h_\delta(r)}{1-h_\delta(r)}.
    \end{equation}

    Let $\mathcal{U}$ be the set of all unlabeled Cayley graphs of $G$, let $\mathcal{U}_{\text{good}}$ be a subset of $\mathcal{U}$ consisting of all unlabeled Cayley graphs $\Cay(G,S)$ with $S\in \calS_{\text{good}}$, and let $\mathcal{U}_{\text{bad}}=\mathcal{U}-\mathcal{U}_{\text{good}}$. Take arbitrary $S_1$ and $S_2$ in $\calS_{\text{good}}$, and denote $\Gamma_1=\Cay(G,S_1)$ and $\Gamma_2=\Cay(G,S_2)$. Then the definition of $\calS_{\text{good}}$ yields that
    \[
        \Aut(\Gamma_1)=R(G)\rtimes\la\iota\ra=\Aut(\Gamma_2).
    \]
    Suppose that $\Gamma_1\cong \Gamma_2$, and let $\varphi\in \Sym(G)$ be a graph isomorphism from $\Gamma_1$ to $\Gamma_2$. Then $\Aut(\Gamma_1)^\varphi=\Aut(\Gamma_2)$, and so $\varphi$ normalizes $R(G)\rtimes\la\iota\ra $. {Since $G$ has exponent greater than $2$, each element of $(R(G)\rtimes\la\iota\ra )-R(G)$ is an involution and $\Center(R(G)\rtimes\la\iota\ra )=R(\calI(G))$. Hence $R(G)$ is the union of $\Center(R(G)\rtimes\la\iota\ra )$ and the set of elements of order greater than $2$ in $R(G)\rtimes\la\iota\ra $, and so $R(G)$ is characteristic in $R(G)\rtimes\la\iota\ra $.} It follows that $R(G)^\varphi=R(G)$, which means $\varphi\in \Hol(G)$. Hence, $S_1$ and $S_2$ are conjugate via an element of $\Hol(G)$. Therefore,
    \[
        |\mathcal{U}_{\text{good}}|\geq \frac{|\calS_{\text{good}}|}{|\Hol(G)|}\geq \frac{|\calS_{\text{good}}|}{2^{(\log_2r)^2+\log_2r}}.
    \]
    Combining this with~\eqref{eqbad} we derive that, for sufficiently large $r$,
    \[
        \frac{|\mathcal{U}_{\text{bad}}|}{|\mathcal{U}|}
        \leq \frac{|\calS_{\text{bad}}|}{|\mathcal{U}_{\text{good}}|}
        \leq \frac{|\calS_{\text{bad}}|}{|\calS_{\text{good}}|}\cdot 2^{(\log_2r)^2+\log_2r}
        \leq \frac{h_\delta(r)}{1-h_\delta(r)}\cdot 2^{(\log_2r)^2+\log_2r}.
    \]
    This completes the proof.
\end{proof}

\begin{proof}[Proof of Corollary~$\ref{corIso1}$]
    {If the exponent of $G$ is greater than $2$, then the result follows from Theorem~\ref{thmIso}, since $k_\delta(r)$ approaches $0$ as $r\to\infty$. Suppose now that the exponent of $G$ is $2$. Then} each Cayley digraph on $G$ is a graph, and the inversion mapping $\iota=1$. Hence, it follows from~\cite[Theorem~1.3(b)]{GSX2025} that the proportion of unlabeled Cayley graphs $\Gamma$ on $G$ with
    \[
    \Aut(D(\Gamma)) = (R(G) \rtimes \langle \iota \rangle) \times C_2
    \]
    approaches $1$ as $r\to \infty$. Recalling
    \[
        R(G)\rtimes\la\iota\ra \leq \Aut(\Cay(G,S)) \ \text{ and }\  \Aut(\Cay(G,S))\times C_2\leq \Aut(D(\Cay(G,S))),
    \]
    we conclude that the proportion of unlabeled Cayley graphs of $G$ that are stable approaches $1$ as $r\to \infty$, and the proportion of unlabeled Cayley graphs $\Gamma$ on $G$ with $\Aut(\Gamma)=R(G)\rtimes\la\iota\ra$ approaches $1$ as $r\to \infty$.
\end{proof}

\section*{Acknowledgements}

The authors thank the organizers of the ``Workshop on Group Actions on Discrete Structures'' (Kranjska Gora, June 2024), where this work was initiated. We acknowledge the support of the Slovenian Research and Innovation Agency under Project~N1~0216 for funding the workshop. We are grateful to Ted Dobson and \DJ or\dj e Mitrovi\'c for valuable discussions during the workshop, and particularly to \DJ or\dj e Mitrovi\'c for reviewing the first draft and suggesting revisions. We also thank Ran Ju, Yu Wang, and Ju Zhang for helpful discussions regarding the presentation of the paper. The work of Zhishuo Zhang was supported by the Melbourne Research Scholarship at The University of Melbourne. Shasha Zheng was supported by the Postdoctoral Fellow Fund of Comenius University in Bratislava.

\section*{Declaration on the use of artificial intelligence}

In preparing this article, the authors used AI-based assistants for two purposes only: to improve the English wording and the presentation of the text, and to assist in checking the proofs. All statements, arguments and proofs were conceived and verified by the authors, and all references were located and checked by the authors, who take full responsibility for the content of the article.


\begin{thebibliography}{99}

\bibitem{AKK2024}
M. Ahanjideh, I. Kov\'acs and K. Kutnar, Stability of rose window graphs, \textit{J. Graph Theory}, {\bf 107} (2024), no.~4, 810--832.

\bibitem{BG1982}
L.~Babai and C.~D.~Godsil, On the automorphism groups of almost all Cayley graphs, \textit{European J. Combin.}, {\bf 3} (1982), no.~1, 9--15.

\bibitem{BHR2013}
J.~N.~Bray, D.~F.~Holt and C.~M.~Roney-Dougal, \textit{The maximal subgroups of the low-dimensional finite classical groups}, Cambridge Univ. Press, Cambridge, 2013.

\bibitem{Bychawski2024}
B. Bychawski, Classification of unstable circulants of square-free order, \href{https://arxiv.org/abs/2410.00701}{arXiv:2410.00701}.

\bibitem{C1979}
Y.~Caro, New results on the independence number, Tech. Report, Tel-Aviv University, 1979.

\bibitem{DSV2016}
E. Dobson, P. Spiga and G. Verret, Cayley graphs on abelian groups, \textit{Combinatorica}, {\bf 36} (2016), no.~4, 371--393.

\bibitem{DX2000}
S.~F.~Du and M.~Y.~Xu, A classification of semisymmetric graphs of order $2pq$, \textit{Comm. Algebra}, {\bf 28} (2000), no.~6, 2685--2715.

\bibitem{FH2022}
B.~Fernandez and A.~Hujdurovi\'{c}, Canonical double covers of circulants, \textit{J. Combin. Theory Ser. B}, {\bf 154} (2022), 49--59.

\bibitem{GSX2025}
Y.~Gan, P.~Spiga and B.~Xia, Asymptotic Enumeration of Haar Graphical Representations, \textit{Combinatorica}, {\bf 45} (2025), no.~5, Paper No. 51.

\bibitem{GLP2004}
M.~Giudici, C.~H.~Li and C.~E.~Praeger, Analysing finite locally $s$-arc transitive graphs, \textit{Trans. Amer. Math. Soc.}, {\bf 356} (2004), no.~1, 291--317.

% \bibitem{Godsil1978}
% C.~D. Godsil, GRRs for nonsolvable groups, in \textit{Algebraic methods in graph theory, Vol. I, II (Szeged, 1978)}, pp. 221--239, Colloq. Math. Soc. J\'anos Bolyai, 25, North-Holland, Amsterdam-New York, 1981.

\bibitem{HK2023}
A. Hujdurovi\'c{} and I. Kov\'acs, Stability of Cayley graphs and Schur rings, \textit{Electron. J. Combin.}, {\bf 32} (2025), no.~2, Paper No. 2.49, 23 pp.

\bibitem{HM2024}
A. Hujdurovi\'c{} and \DJ.~Mitrovi\'c, Some conditions implying stability of graphs, \textit{J. Graph Theory}, {\bf 105} (2024), no.~1, 98--109.

\bibitem{HMM2021}
A.~Hujdurovi\'{c}, \DJ.~Mitrovi\'{c} and D.~W.~Morris, On automorphisms of the double cover of a circulant graph, \textit{Electron. J. Combin.}, {\bf 28} (2021), no.~2, Paper No. 4.43.

\bibitem{HMM2023}
A.~Hujdurovi\'{c}, \DJ.~Mitrovi\'{c} and D.~W. Morris, Automorphisms of the double cover of a circulant graph of valency at most 7, \textit{Algebr. Comb.}, {\bf 6} (2023), no.~5, 1235--1271.

\bibitem{H1967}
B. Huppert, \textit{Finite groups I}, translated by C.~A. Schroeder, Grundlehren der mathematischen Wissenschaften, 364, Springer, Cham, 2025.

\bibitem{Imrich1976}
W. Imrich and M.~E. Watkins,
{On automorphism groups of Cayley graphs},
\textit{Period. Math. Hungar.} {\bf 7} (3-4) (1976), 243--258.

\bibitem{MSZ1989}
D.~Maru\v{s}i\v{c}, R.~Scapellato and N.~Zagaglia Salvi, A characterization of particular symmetric $(0,1)$ matrices, \textit{Linear Algebra Appl.}, {\bf 119} (1989), 153--162.

\bibitem{Morris2021}
D.~W.~Morris, On automorphisms of direct products of Cayley graphs on abelian groups, \textit{Electron. J. Combin.}, {\bf 28} (2021), no.~3, Paper No. 3.5.

\bibitem{DW2023}
D.~W.~Morris, Automorphisms of the canonical double cover of a toroidal grid, \href{https://arxiv.org/abs/2301.05396}{arXiv:2301.05396}.

\bibitem{MMS2022}
J.~Morris, M.~Moscatiello and P.~Spiga, On the asymptotic enumeration of Cayley graphs, \textit{Ann. Mat. Pura Appl. (4)}, {\bf 201} (2022), no.~3, 1417--1461.

% \bibitem{MS2021}
% J.~Morris and P. Spiga, Asymptotic enumeration of Cayley digraphs, \textit{Israel J. Math.}, {\bf 242} (2021), no.~1, 401--459.

\bibitem{MSV2015}
J.~Morris, P. Spiga and G. Verret, Automorphisms of Cayley graphs on generalised dicyclic groups, \textit{European J. Combin.}, {\bf 43} (2015), 68--81.

\bibitem{QXZ2019}
Y.-L.~Qin, B.~Xia and S.~Zhou, Stability of circulant graphs, \textit{J. Combin. Theory Ser. B}, {\bf 136} (2019), 154--169.

\bibitem{QXZ2018}
Y.-L. Qin, B. Xia and S. Zhou, Canonical double covers of generalized Petersen graphs, and double generalized Petersen graphs, \textit{J. Graph Theory}, {\bf 97} (2021), no.~1, 70--81.

\bibitem{S2021}
P.~Spiga, On the equivalence between a conjecture of Babai-Godsil and a conjecture of Xu concerning the enumeration of Cayley graphs, \textit{Art Discrete Appl. Math.}, {\bf 4} (2021), no.~1, Paper No. 1.10, 18 pp.

\bibitem{S2001}
D.~B. Surowski, Stability of arc-transitive graphs, \textit{J. Graph Theory}, {\bf 38} (2001), no.~2, 95--110.

\bibitem{T1941}
P.~Tur\'{a}n, On an extremal problem in graph theory, \textit{Mat. Fiz. Lapok}, {\bf 48} (1941), 436--452.

\bibitem{W1981}
V.~K.~Wei, A lower bound on the stability number of a simple graph, \textit{Bell Laboratories Technical Memorandum}, (1981), No. 81-11217-9.

\bibitem{SW2008}
S.~E. Wilson, Unexpected symmetries in unstable graphs, \textit{J. Combin. Theory Ser. B}, {\bf 98} (2008), no.~2, 359--383.

\bibitem{XZ2023}
B.~Xia and S.~Zheng, Asymptotic enumeration of graphical regular representations, \textit{Proc. Lond. Math. Soc. (3)}, {\bf 127} (2023), no.~5, 1424--1450.

\bibitem{Zhang2025}
J. Zhang, Graph product and the stability of circulant graphs, \href{https://arxiv.org/abs/2504.05721}{arXiv:2504.05721}.

\end{thebibliography}
\end{document}